\documentclass[11pt]{article}
\usepackage[utf8]{inputenc}
\usepackage{amssymb}
\usepackage{amsmath,bbm,mathabx,cases,booktabs,amsthm}
\usepackage{enumitem}   

\usepackage{mathrsfs}
\usepackage{caption,subcaption}
\usepackage{graphicx,multirow}
\usepackage{authblk}
\usepackage{comment}
\usepackage[margin=1.3in]{geometry}

\usepackage{titlesec}
\titleformat{\subsubsection}[runin]
       {\normalfont\bfseries}
       {\thesubsubsection}
       {0.5em}
       {}
       [.]

\newtheorem{theorem}{Theorem}[section]

\newtheorem{proposition}{Proposition}[section]
\theoremstyle{definition}

\theoremstyle{remark}
\newtheorem{remark}[theorem]{Remark}
\numberwithin{equation}{section}

\DeclareMathOperator*{\argmin}{arg\,min}

\newcommand{\sgn}{\text{sgn}}
\newcommand{\subscript}[2]{$#1 _ #2$}

\begin{document}

\date{}
\title{\vspace{-1em}Covariance models and Gaussian process regression for the wave equation. Application to related inverse problems}

\author[a]{I. Henderson \thanks{Corresponding author. E-mail address: \texttt{henderso@insa-toulouse.fr}}}
\author[a]{P. Noble}
\author[a]{O. Roustant}

\affil[a]{Institut de Mathématiques de Toulouse; UMR 5219\\ 
Université de Toulouse; CNRS \\
INSA, F-31077 Toulouse, France}


\maketitle

\begin{abstract}
In this article, we consider the general task of performing Gaussian process regression (GPR) on pointwise observations of solutions of the 3 dimensional homogeneous free space wave equation.
In a recent article, we obtained promising covariance expressions tailored to this equation: we now explore the potential applications of these formulas.
We first study the particular cases of stationarity and radial symmetry, for which significant simplifications arise. 
We next show that the true-angle multilateration method for point source localization, as used in GPS systems, is naturally recovered by our GPR formulas in the limit of the small source radius. Additionally, we show that this GPR framework provides a new answer to the ill-posed inverse problem of reconstructing initial conditions for the wave equation from a limited number of sensors, and simultaneously enables the inference of physical parameters from these data. We finish by illustrating this ``physics informed'' GPR on a number of practical examples.
\end{abstract}

{\footnotesize Keywwords: wave equation, covariance models, Gaussian processes, Gaussian process regression, physical parameter estimation, initial condition reconstruction.}

\section{Introduction}
Machine learning techniques have proved time and again that they can provide efficient solutions to difficult problems in the presence of field data. A key element to this success is the incorporation of ``expert knowledge" in the corresponding statistical models. In many practical applications, this knowledge takes the form of mathematical models which are sometimes already well understood. This is e.g. common when dealing with problems arising from physics, in which case the mathematical models often take the form of Partial Differential Equations (PDEs), such as the wave equation at hand in this article. Because of the broadness of the applications PDEs offer, large efforts have been devoted to studying and solving them, both theoretically \cite{evans1998} and numerically \cite{grossmann2007numerical}. These equations impose very specific (yet often simple) structures on the observed data which can be very difficult to capture or mimic with general machine learning models.

In this article, we will focus on the linear 3 dimensional homogeneous free space wave equation. This equation is the prototype for describing simple 3D phenomena which propagate at finite speed; although particularly simple in the landscape of PDEs, it is in fact central for many applications emerging from different fields such as acoustics or electromagnetics. The homogeneity assumption is also commonly encountered in physics, when modelling conservation laws.
Given that the main structures of the solutions of this PDE are well known, one may thus attempt to incorporate them in the machine learning models that work with such solutions. 

The class of models we will deal with is that of Gaussian Process Regression (GPR), which is a Bayesian framework for function regression and interpolation \cite{gpml2006}. It is especially adapted to performing inference in the presence of limited/scattered data, say measurements from a small number of scattered sensors. It is also a ``kernel method", meaning that it is built upon a positive semidefinite function, the kernel in question. In the language of Bayesian inference, GPR puts a \textit{prior} probability distribution on a suitable function space in which the unknown function $u$ is assumed to lie. This prior is then conditioned on available field data involving $u$ thanks to Bayes' law, which in turn provides a \textit{posterior} probability distribution from which statistical estimators related to $u$ can be computed. The posterior expectation in particular plays the role of an approximant of $u$ while the posterior covariance provides posterior error bounds. In the case of GPR, these prior and posterior distributions are in many ways generalizations to infinite dimensions of the multivariate normal distribution, and are fully specified by a mean and covariance functions. 
These priors are naturally obtained by modelling $u$ as a {sample path} of a {Gaussian process} and we will thus say that we put a Gaussian process (GP) prior over $u$. Imposing strict {linear} constraints on a GP prior as well as on the posterior expectation it provides is straightforward in principle; we will apply this observation to the case where the linear constraint is the homogeneous wave equation itself, as in \cite{hnr_bernoulli}.

Thus, we will first be concerned with building GP priors which incorporate \textit{beforehand} the knowledge that the sought function is in fact a solution to the wave equation, thus drastically lowering the dimension of the function space upon which the prior is set. In practice, the main consequence will be that all the possible estimators of $u$ provided by GPR will also be solutions to the same wave equation. Nevertheless, from a random field perspective, it is remarkable that this property will in fact also hold at the level of the sample paths of the GP, when the PDE is understood in the distributional sense (\cite{hnr_bernoulli}, Proposition 4.1). Those covariance formulas are particular cases of general ones first described in \cite{hnr_bernoulli}, which take the form of multidimensional convolutions against the PDE's Green's function. They were derived by putting generic Gaussian process priors over the initial conditions of the wave equation and propagating them through the solution map of the said equation, leading to ``wave equation-tailored" covariance functions. Though interesting for theoretical purposes, these convolutions are very expensive to evaluate numerically, which constitutes a limitation for their use in GPR. In this article, we explore the particular cases where the initial condition priors are either stationary (Proposition \ref{prop: F_t * F_tp}) or radially symmetric (Proposition \ref{prop: radial sym kernel}), as then notable simplifications can be obtained. We then study the case of point sources, for which we show that the task of recovering the position of the point source using multilateration (as e.g. in GPS systems, see \cite{trilateration_fang}) is unexpectedly recovered by maximizing the likelihood attached to the GPR models we previously obtained for the wave equation, in the limit of the small source radius (Figure \ref{fig: point source}). 
We will also discuss applications in physical parameter estimation and initial condition reconstruction. Recovering the initial position in particular is the purpose of photoacoustic tomography (PAT, \cite{ammari2012}, Chapter 3), an exercise for which we will provide a simple proof of concept application, in the presence of radial symmetry.

\paragraph{Related literature}
The idea of solving and ``learning" linear ODEs and PDEs thanks to GPR probably goes back to \cite{graepel2003} and has been re-explored ever since. A large part of the subsequent works inspired by \cite{graepel2003} deal with PDEs of the form $L(u) = f$ 
where $f$ is a partially known \textit{interior} source term: that is, $f$ and $u$ have the same input space. We will not be interested in this case as we will impose the strict condition that $f\equiv 0$, as is e.g. the case in PAT. In our case, the initial conditions will instead play the role of the source terms. For dealing with interior source terms, see \cite{sarkka2011,alvarez2009,alvarez2013,Srkk2019GaussianPL,LpezLopera2021PhysicallyInspiredGP,raissi2017,raissi2018numericalGP} and \cite{alvarado2014,alvarado2016} for subsequent applications to inhomogeneous wave equations. 
See also \cite{CHEN_owhadi_2021} for an alternative method applicable to nonlinear PDEs.
Compared to these approaches, ensuring (deterministically) the homogeneity constraint $f=0$ in the wave equation will allow us to drastically reduce the dimensionality of the problem of approximating $u$ given scattered measurements of $u$.

Ensuring homogeneous PDE constraints on centered GPs is done by appropriately constraining its covariance kernel (\cite{hnr_bernoulli}, Proposition 3.5).
Such PDE constrained kernels have been explicitly built for a number of classical PDEs, namely: divergence-free vector fields \cite{Narcowich1994GeneralizedHI,scheuerer2012}, curl-free vector fields \cite{fuselier2007refined,scheuerer2012,wahlstrom2013,jidling2017}, the Laplace equation \cite{Schaback2009SolvingTL,mendes2012,albert2020}, Maxwell's equations \cite{hegerman2018}, the heat equation in 1D \cite{albert2020} and 2D \cite{ginsbourger2016}, Helmholtz' 2D equations \cite{albert2020}, and linear solid mechanics \cite{Jidling2018ProbabilisticMA}. See also \cite{vergara2022general} where generic PDE-constrained kernels are built under stationarity assumptions. For further discussions and references on PDE constrained random fields, we refer to \cite{hnr_bernoulli}, Section 1. This article is the continuation of a previous work \cite{hnr_bernoulli}, where we described a covariance kernel tailored to the wave equation at hand in this article. In parallel with homoegeneous PDEs,
\cite{hegermann2021LinearlyCG,gulian2022,pmlr-v89-solin19a} enforce homogeneous \textit{boundary conditions} on the covariance kernel. We finish by mentioning that fine properties of a stochastic three dimensional wave equation are studied in \cite{dalang_2009}. The wave equation in \cite{dalang_2009} is not homogeneous, and because of the nonlinearity they consider, a precise investigation of the covariance function of the solution process is not considered.

The approach presented in this article falls in the field of Bayesian methods for solving PDE related inverse problems, the literature of which is extensive; see \cite{stuart_2010,dashti2013map,dashti_stuart_pde,Dashti2017} and the many references therein. However, the method we adopt here differs from the standard Bayesian inversion methods aforementioned in that we incorporate the PDE constraint \textit{beforehand}, i.e. directly in the prior; the PDE does not only appear in the likelihood. See \cite{owhadi_bayes_homog} for a point of view similar with that of the present article, which uses PDE-tailored GP priors for building optimal finite dimensional approximations of solution spaces of elliptic PDEs. 

The inverse problems we will study deal with approximating the initial conditions of \eqref{eq:wave_eq} as well as the related physical parameters (wave speed, source location and source size), given scattered measurements of the solution $u$. 
A general methodology for estimating the parameters of a linear PDE using GPR is described in \cite{raissi2017}, using the forward differential operator. Here we will rather use its inverse, i.e. the Green's function.
The task of approximating the initial position in particular is the purpose of photoacoustic tomography (PAT), which is a technique commonly used e.g. in biomedical imaging \cite{ammari2012}. See e.g. \cite{Kuchment2015, anastasio_2007} for details on the standard mathematical techniques and models used in PAT. Note that the solution is often assumed available on a surface enclosing the source \cite{xu2005universal}, in order to use Radon transforms or similar inversion formulas. 
Our method instead allows the sensors to be arbitrarily scattered. As the corresponding PAT problem becomes ill-posed, we do not aim for a full reconstruction of the initial conditions. Instead, we show that our method amounts to computing an orthogonal projection of the solution over a well-chosen finite dimensional space. Of course, the geometry of the sensor locations plays a crucial role in the accuracy of our model, but the reconstruction formula we introduce remains nonetheless independent of any underlying geometry assumptions. In the two dimensional setting, it is worth noting that \cite{Purisha_2019} already showed that a GPR methodology based on Radon transforms could be set up for solving x-ray tomography problems in the presence of limited (scattered) data.  

\paragraph{Organization of the paper} For self-containment, section 2 is dedicated to reminders on (Gaussian) random fields and GPR. Section 3 is dedicated to the study of GP priors tailored to the wave equation. In section 4, we showcase some numerical applications of the previous section on wave equation data. We conclude in section 5. For the sake of readability, all the proofs as well as technical definitions concerning convolutions and tensor products are gathered in the appendix.

\paragraph{Notations}
Let $\mathcal{D}$ be a set, $m: \mathcal{D} \rightarrow \mathbb{R}$ and $k: \mathcal{D} \times \mathcal{D} \rightarrow \mathbb{R}$. Given $x\in \mathcal{D}$, $k_x$ denotes the function $y\mapsto k(x,y)$. If $X = (x_1,...,x_n)^T$ is a column vector in $\mathcal{D}^n$, we denote $m(X)$ the column vector such that $m(X)_i = m(x_i)$, $k(X,X)$ the square matrix such that $k(X,X)_{ij} = k(x_i,x_j)$ and given $x \in \mathcal{D}$, $k(X,x)$ the column vector such that $k(X,x)_i = k(x_i,x)$. 
The variables $(r,\theta,\phi), r\geq 0, \theta \in [0,\pi], \phi\in [0,2\pi]$, denote spherical coordinates and $S$ denotes the unit sphere of $\mathbb{R}^3$. We write $d\Omega = \sin \theta d\theta d\phi$ its surface differential element; $\gamma = (\sin \theta \cos \phi,\sin \theta \sin \phi,  \cos \theta)^T \in S$ denotes the unit length vector parametrized by $(\theta,\phi)$.

\section{Background on Gaussian process regression}
 
\subsection{Random fields, Gaussian processes, positive semidefinite functions}\label{sub: GP}
Let $\mathcal{D}$ be a set. A random field $(U(x))_{x \in \mathcal{D}}$ is a collection of random variables defined on the same probability space $(\Omega,\mathcal{F},\mathbb{P})$. It is second order if for all $x\in\mathcal{D},\ \mathbb{E}[U(x)^2] < +\infty.$ Its sample paths are the deterministic functions $x\mapsto U(x)(\omega)$, given $\omega\in\Omega$.
$(U(x))_{x \in \mathcal{D}}$ is a GP if for all $(x_1,...,x_n) \in \mathcal{D}^n$, the law of $(U({x_1}),...,U({x_n}))^T$ is a multivariate normal distribution.
The law of a GP is characterized by its mean and covariance functions (\cite{janson_1997}, Section 8), defined by  $m(x):= \mathbb{E}[U(x)]$ and $k(x,x') = \text{Cov}(U(x),U({x'})) = \mathbb{E}[U(x)U({x'})] - m(x)m(x')$, and we write $(U(x))_{x \in \mathcal{D}} \sim GP(m,k)$. Given $\omega\in\Omega$, the associated sample path is the deterministic function $U_{\omega}: x \mapsto U(x)(\omega)$.
The mean function can be chosen arbitrarily, but the covariance function has to be symmetric and positive semidefinite, which means that for all $(x_1,...,x_n) \in \mathcal{D}^n$, the matrix $(k(x_i,x_j))_{1\leq i,j\leq n}$ is symmetric non negative definite (\cite{gpml2006}, Section 4.1). In the rest of the paper, positive semidefinite functions will implicitly be assumed symmetric.
The mathematical properties of the GP are encoded in the function $k$. Furthermore, there is a bijection between positive semidefinite functions and covariance functions of centered GPs (\cite{janson_1997}, Theorem 8.2). We will thus focus on the design of positive semidefinite kernels.
A covariance kernels is \textit{stationary} if $k(x,x')$ only depends on the increment $x-x'$: $k(x,x') = k_S(x-x')$ for some function $k_S$. Common examples of stationary kernels are the squared exponential and Matérn kernels \cite{gpml2006}; see equation \eqref{eq:matern 5/2}. Informally, if the covariance function of a GP is stationary, then its sample paths ``look similar at all locations'' (\cite{gpml2006}, p.4).

\subsection{Gaussian process regression \cite{gpml2006}}\label{sub: GPR}
\subsubsection{Kriging equations} 
GPs can be used for function interpolation. Let $u$ be  a function defined on $\mathcal{D}$ for which we know a dataset of values $B = \{u(x_1),...,u(x_n)\}$. Conditioning the law of a GP $(U(x))_{x \in \mathcal{D}} \sim GP(m,k)$ on the data $B$ yields a second GP defined by
$ V(x):= (U(x) | U({x_i}) = u(x_i), i = 1,...,n)$.
Its mean and covariance functions $\tilde{m}$ and $\tilde{k}$ are given by the so-called \textit{Kriging} equations \eqref{eq:krig mean} and \eqref{eq:krig cov}. Note $X = (x_1,...,x_n)^T$ and assume that $K(X,X)$ is invertible, then \cite{gpml2006}
\begin{numcases}{}
    \Tilde{m}(x) &= \hspace{3mm}$m(x) + k(X,x)^Tk(X,X)^{-1}(u(X) - m(X)),$ \label{eq:krig mean} \\
    \Tilde{k}(x,x') &= \hspace{3mm}$k(x,x') - k(X,x)^Tk(X,X)^{-1}k(X,x').$\label{eq:krig cov}
\end{numcases}
The function $\Tilde{m}$ is an estimator of $u$ and for all $x$ in $\mathcal{D}$, $\tilde{m}(x)$ can be used for predicting the value $u(x)$. By construction, for all observation points $x_i$, we have $\Tilde{m}(x_i) = u(x_i)$ and $\Tilde{k}(x_i,x_i) = 0$. If observing noisy data ${U_i = U(x_i) + \varepsilon_i}$ with $(\varepsilon_1,...,\varepsilon_n)^T \sim \mathcal{N}(0,\sigma^2I_n)$ independent from $U$, one replaces $K(X,X)$ with $K(X,X) + \sigma^2 I$ in the Kriging equations and leaves the other terms $k(X,x)$ unchanged. This amounts to applying Tikhonov regularization on $k(X,X)$, which is also relevant for approximating equations \eqref{eq:krig mean} and \eqref{eq:krig cov}  when $k(X,X)$ is ill-conditioned.

\subsubsection{Tuning covariance kernels \cite{gpml2006}}\label{subsub:tuning_cov}
Covariance functions are usually chosen among a parametrized family of kernels $\{k_{\theta}, \theta \in \Theta \subset \mathbb{R}^q\}$. $\theta$ contains the \textit{hyperparameters} of $k_{\theta}$. One then attempts to find the value $\theta$ which fits best the observations $u_{obs} = (u_1,...,u_n)^T$, the set of observations of $u$ at locations $X = (x_1,...,x_n)$. This is performed by maximizing the \textit{marginal likelihood}, which is the probability density of the random vector $(U(x_1),...,U(x_n))^T$ at point $u_{obs}$, given $\theta$. Denote $p(u_{obs}|\theta)$ the associated marginal likelihood at $\theta$, one searches for $\hat{\theta}$ such that $\hat{\theta} = \text{argmax}_{\theta \in \Theta} p(u_{obs} | \theta)$.
Explicitly, assuming that $m \equiv 0$, then we have $(U(x_1),...,U(x_n))^T \sim \mathcal{N}(0,k_{\theta}(X,X))$ and
\begin{align}\label{eq:likelihood}
p(u_{obs} | \theta) = \frac{1}{(2\pi)^{n/2}\det k_{\theta}(X,X)^{1/2}}e^{-\frac{1}{2} u_{obs}^T k_{\theta}(X,X)^{-1}u_{obs}}.
\end{align}
Equivalently, for noisy observations with identical noise standard deviation $\sigma$, set 
\begin{align}
    \mathcal{L}(\theta,\sigma^2):&= -2\log p(u_{obs} | \theta) - n\log 2\pi \nonumber \\
    &= u_{\mathrm{obs}}^T (k_{\theta}(X,X)+ \sigma^2 I_n)^{-1}u_{\mathrm{obs}} + \log \det (k_{\theta}(X,X) + \sigma^2 I_n).
\end{align}
We call $\mathcal{L}(\theta,\sigma^2)$ the negative log marginal likelihood, and one may rather attempt to find $\hat{\theta}$ such that $\hat{\theta} = \argmin_{\theta \in \Theta} \mathcal{L}(\theta,\sigma^2)$.
Note that $\sigma$ can also be interpreted as a hyperparameter and estimated through negative log marginal likelihood minimization.
\subsubsection{The RKHS point of view}\label{subsub:rkhs}
The Kriging equations \eqref{eq:krig mean} and \eqref{eq:krig cov} can alternatively be viewed as orthogonal projections of $u$ in a suitable Hilbert space. 
Given a positive semidefinite kernel $k$ defined on a set $\mathcal{D}$, one may build a Reproducing Kernel Hilbert Space (RKHS) of functions defined on $\mathcal{D}$, which we denote by $\mathcal{H}_k$ (\cite{agnan2004}, Theorem 3). The inner product of $\mathcal{H}_k$ verifies the reproducing property \cite{wendland2004scattered}: $\langle k(x,\cdot),k(x',\cdot) \rangle_{\mathcal{H}_k} = k(x,x')$. One may then formulate the following regularized interpolation problem \cite{fasshauer2007,wendland2004scattered}
\begin{align}\label{eq:rkhs pov pb}
    \inf_{v \in \mathcal{H}_k}||v||_{\mathcal{H}_k} \hspace{5mm} \text{s.t.} \hspace{5mm} v(x_i) = u(x_i) \ \ \ \forall i \in \{1,...,n\}.
\end{align}
Then $\tilde{m}$ in equation \eqref{eq:krig mean} is the unique solution of \eqref{eq:rkhs pov pb}. One can also show \cite{wendland2004scattered} that equation \eqref{eq:krig mean} amounts to $\Tilde{m} = m + p_F(u-m)$,
where $p_F$ stands for the orthogonal projection operator on $F:= \text{Span}(k(x_1,\cdot),...,k(x_n,\cdot))$  with reference to the inner product of $\mathcal{H}_k$. If in particular $m\equiv 0$, then $\Tilde{m} = p_F(u)$. Likewise, equation \eqref{eq:krig cov} amounts to $\Tilde{k}(x,\cdot) = P_{F^\perp}(k(x,\cdot))$. 
Viewing the Kriging mean as an orthogonal projection over a finite dimensional deterministic space is reminiscent of Fourier series or Galerkin reconstruction approaches.

\section{Gaussian process priors for the 3D wave equation}\label{section:GP_wave}

\subsection{General solution to the wave equation}
Denote the 3D Laplace operator $\Delta = \partial_{xx}^2 + \partial_{yy}^2 + \partial_{zz}^2$ and the d'Alembert operator with the box symbol, $\Box = c^{-2} \partial_{tt}^2 - \Delta$ with constant wave speed $c > 0$.
Consider then the following initial value problem in the free space $\mathbb{R}^3$
\begin{align}\label{eq:wave_eq}
\begin{cases}
    \hfil \Box w &= 0 \hspace{120pt} \forall (x,t) \in \mathbb{R}^3 \times \mathbb{R}_+^*,  \\
    \hfil w(x,0) &= u_0(x), \hspace{6pt} \partial_t w(x,0) = v_0(x) \hspace{15pt} \forall x \in \mathbb{R}^3.
\end{cases}
\end{align}
The solution of this problem is unique in the distributional sense (\cite{Duistermaat2010}, p. 164). It can be extended to all $t\in \mathbb{R}$ and is represented as follow (\cite{Duistermaat2010}, p. 295)
\begin{align}\label{eq sol wave}
    w(x,t) = (F_t * v_0)(x) + (\Dot{F}_t * u_0)(x), \hspace{10mm} \forall (x,t) \in \mathbb{R}^3 \times \mathbb{R}.
\end{align}
$(F_t)_{t\in\mathbb{R}}$ is the Green's function of the wave equation (\cite{Duffy2015GreensFW}, p. 202). For fixed $t$, $F_t$ is a singular measure, meaning that it has no density with reference to the Lebesgue measure. $\Dot{F}_t$ is $F_t$'s ``time derivative" (formally, $\dot{F}_t = \partial_tF_t$, \cite{Duistermaat2010}, equation (18.16) p. 297), understood as a continuous linear form over $C^1(\mathbb{R}^3)$. Details on the definition of the convolution $F_t*v_0$ are given in Section \ref{sub:conv_measure}, while $\dot{F}_t*u_0$ is effectively computed as $\dot{F}_t*u_0=\partial_t(F_t*v_0)$. Explicitly, $F_t$ and $\dot{F}_t$ are defined by
\begin{align}\label{eq:ft ftp in 3D}
    F_t = \frac{\sigma_{c|t|}}{4\pi c^2 t}, \hspace{8mm} \text{and} \hspace{8mm} \forall \varphi\in C^1(\mathbb{R}^3),\  \ \langle\Dot{F}_t,\varphi\rangle = \partial_t\bigg( \int_{\mathbb{R}^3} \varphi(x)F_t(dx)\bigg),
\end{align}
where $\sigma_R$ is the surface measure of the sphere of center $0$ and radius $R$, and $\langle \cdot,\cdot\rangle$ is the duality bracket between $C^1(\mathbb{R}^3)$ and its dual.
If $u_0 \in \mathcal{C}^1(\mathbb{R}^3) $ and $v_0 \in \mathcal{C}^0(\mathbb{R}^3)$, then $w$ as defined in \eqref{eq sol wave} is a pointwise defined function and equation \eqref{eq sol wave} reduces to the Kirschoff formula (\cite{evans1998}, p. 72), which writes in spherical coordinates:
\begin{align}\label{eq:kirschoff}
w(x,t) = \int_{S}tv_0(x-c|t|\gamma) + u_0(x-c|t|\gamma) -  c|t|\gamma \cdot \nabla u_0(x-c|t|\gamma) \frac{d\Omega}{4\pi}.
\end{align}
 
\subsection{Gaussian process priors for the wave equation}
\subsubsection{General covariance structure}\label{subsub:gen_formula}
Suppose that the initial conditions $u_0$ and $v_0$ are realizations of two independent centered Gaussian processes, $U^0 \sim GP(0,k_{\mathrm{u}})$ and $V^0 \sim GP(0,k_{\mathrm{v}})$. That is, $u_0 = U_{\omega}^0$ and $v_0 = V_{\omega}^0$ for some $\omega\in\Omega$. This assumption is relevant e.g. when $u_0$ and $v_0$ are unknown, in which case $U^0$ and $V^0$ are interpreted as GP priors over $u_0$ and $v_0$. We will assume that the sample paths of $V^0$ are continuous and that of $U^0$ are continuously differentiable, in order to use the formula \eqref{eq:kirschoff} (see \cite{hnr_bernoulli}, Section 4.2 for more details and discussions on these assumptions).
By solving \eqref{eq:wave_eq}, one obtains a time-space random field $W(x,t)$ defined by 
\begin{align}\label{eq:W sol}
W(x,t): \Omega \ni \omega \longmapsto (F_t * V^0_{\omega})(x) +  (\Dot{F}_t * U^0_{\omega})(x).  
\end{align}
The next result, which describes the covariance function of $W$, is the starting point of this paper.
\begin{proposition}[\cite{hnr_bernoulli}, Proposition 4.1]\label{prop: wave kernel} Denote $z = (x,t)$ and $z' = (x',t')$ the space-time variables. Let $k_{\mathrm{u}}$ (resp. $k_{\mathrm{v}}$) be a positive semidefinite function such that the sample paths of the associated GP are continuously differentiable (resp. continuous). In particular, $k_{\mathrm{v}}\in C^0(\mathbb{R}^3\times\mathbb{R}^3)$ and $ k_{\mathrm{u}}(x,\, .), k_{\mathrm{u}}(.\,,x')\in C^{1}(\mathbb{R}^3)$ for all $x,x'\in\mathbb{R}^3$. Define then the two functions
\begin{align}
{k_{\mathrm{v}}^{\mathrm{wave}}(z,z') =[(F_t \otimes F_{t'}) * k_{\mathrm{v}}](x,x')},\label{eq: kv wave} \\
{k_{\mathrm{u}}^{\mathrm{wave}}(z,z') = [(\Dot{F}_t \otimes \Dot{F}_{t'}) * k_{\mathrm{u}}](x,x')}. \label{eq: ku wave}
\end{align}
$(i)$ Then $(W(z))_{z \in \mathbb{R}^3 \times \mathbb{R}}$ is a centered GP whose covariance kernel is given by
\begin{align}\label{eq:wave kernel}
k_{\mathrm{w}}(z,z') = k_{\mathrm{v}}^{\mathrm{wave}}(z,z') + k_{\mathrm{u}}^{\mathrm{wave}}(z,z').
\end{align}
(ii) Conversely, any centered second order random field with a.s. continuous sample paths and with covariance function $k_{W}$ has its sample paths solution of the wave equation \eqref{eq:wave_eq} for some $u_0$ and $v_0$, in the sense of distributions, almost surely.
\end{proposition}
Equation \eqref{eq: kv wave} is to be understood in the sense of the appendix section \ref{sub:conv_measure}, while in practice, equation \eqref{eq: ku wave} can be computed as $[(\Dot{F}_t \otimes \Dot{F}_{t'})*k_{\mathrm{u}}](x,x') = \partial_t \partial_{t'} [(F_t \otimes F_{t'})*k_{\mathrm{u}}](x,x')$.
The proof of equation \eqref{eq:wave kernel} relies on Fubini's theorem, to permute $\mathbb{E}[\cdot]$ and integrals over the sphere $S$ (see equation \eqref{eq:kirschoff}). To apply Fubini's theorem, one needs the maps $(x,\omega) \mapsto V(x)(\omega)$ and $(x,\omega) \mapsto \partial_{x_i}U(x)(\omega), i \in \{1,2,3\}$ to be measurable. In our case this property holds, up to a modification, because the random fields $V$ and $\partial_{x_i}U(x)$ are assumed a.s. continuous (see \cite{hnr_bernoulli}, Section 2.1.2 for further discussions). Complete expressions of equations \eqref{eq: kv wave} and \eqref{eq: ku wave} in terms of integrals of $k_{\mathrm{u}}$, its first derivatives and $k_{\mathrm{v}}$ over the unit sphere can be found in \cite{hnr_bernoulli}, p. 23. They are derived from the Kirschoff formula \eqref{eq:kirschoff}.

\begin{remark}\label{rk:drop_gp}
A more general result holds if one drops the GP assumption over $(V^0(x))_{x\in\mathbb{R}^3}$ and $(U^0(x))_{x\in\mathbb{R}^3}$. If we only assume that $V^0$ (resp. $U^0$) is a centered second order random field with a.s. continuous (resp. a.s. continuously differentiable) sample paths and covariance function $k_{\mathrm{v}}$ (resp. $k_{\mathrm{u}}$), then $W$ in equation \eqref{eq:W sol} is well-defined, centered, and its covariance function is $k_{\mathrm{w}}$ in equation \eqref{eq:wave kernel}. Only the Gaussianity of $W$ is lost. Indeed, the proof of Proposition 4.1, \cite{hnr_bernoulli}, only uses the aforementioned relaxed assumptions over $U^0$ and $V^0$ to obtain the formula \eqref{eq:wave kernel}. The Gaussianity of $U^0$ and $V^0$ is only used to show that $W$ is also a GP. Non Gaussian (say log normal or exponential) priors are relevant e.g. for modelling nonnegative initial conditions. They are especially interesting for the wave equation because the nonnegativity of the measure $F_t$ yields the following remarkable positivity preserving property: if $u_0 = 0$ and $v_0 \geq 0$, then $w$ in equation \eqref{eq sol wave} verifies $w(x,t) \geq 0$ for all $t\geq 0$.
\end{remark}
Observe now that for all $z = (x,t) \in \mathbb{R}^3 \times \mathbb{R}$, we have $\Box k_{\mathrm{w}}(z,\cdot) = 0$. Using equation \eqref{eq:krig mean}, one then deduces that all the Kriging mean obtained using the kernel $k_{\mathrm{w}}$ always verifies $\Box\tilde{m} = 0$.
For this reason, we call WIGPR (``Wave equation informed GPR") the act of performing GPR with a covariance kernel of the form \eqref{eq:wave kernel}. Note that the inheritance of the distributional PDE constraint over the sample paths of the conditioned GP is proved in \cite{hnr_bernoulli}, Proposition 3.8.

In applications, a first obstacle of WIGPR is the cost of the evaluation of expressions \eqref{eq: kv wave} and \eqref{eq: ku wave}, both in computational resources and in memory. Indeed, their computation requires $4$-dimensional convolutions. This motivates the study of special cases of expressions \eqref{eq: kv wave} and \eqref{eq: ku wave}. In the next paragraphs, we focus on stationarity and radial symmetry assumptions.

\subsubsection{Stationary initial conditions} Many standard covariance kernels used for GPR are stationary \cite{gpml2006}. More generally, a centered second order stochastic process is said to be \textit{stationary in the wide (or weak) sense} if its covariance function is stationary (\cite{gpml2006}, footnote 2 p. 79). Such stochastic processes play a central role in many different fields such as time series analysis or signal processing \cite{hamilton1994}. Because of the popularity of such stationary random field models as well as GPR methods based on stationary kernels, we study equation \eqref{eq: kv wave} when $k_{\mathrm{v}}$ is stationary. For conciseness, we restrict ourselves to the case where $u_0 = 0$, i.e. $k_{\mathrm{u}} = 0$.
\begin{proposition}\label{prop: F_t * F_tp} Assume that $k_{\mathrm{v}}$ is continuous and stationary: $k_{\mathrm{v}}(x,x') = k_S(x-x')$. \\
(i) Then $k_{\mathrm{v}}^{\mathrm{wave}}$ is stationary in space and
\begin{align}\label{eq:k_v_statio}
[(F_t \otimes F_{t'}) * k_{\mathrm{v}}](x,x') = (F_t * F_{t'} * k_S)(x-x').
\end{align}
(ii) Moreover, the measure $F_t * F_{t'}$ is absolutely continuous over $\mathbb{R}^3$. Denoting $|h|$ the Euclidean norm of $h \in \mathbb{R}^3$ and identifying $F_t * F_{t'}$ with its density, we have
\begin{align} \label{f_t star f_t'}
    (F_t * F_{t'})(h) = \frac{\sgn(t)\sgn(t')}{8\pi c^2|h|}\mathbbm{1}_{\big[c\big||t|-|t'|\big|, c(|t| + |t'|)\big]}(|h|).
\end{align}
\end{proposition}
If $k_{\mathrm{u}}$ is assumed zero, and if $V^0$ only satisfies the minimal assumptions of Remark \ref{rk:drop_gp} as well as wide sense stationarity, then the covariance expression \eqref{eq:k_v_statio} still holds for the solution process $W$ in equation \eqref{eq:W sol}.
Formally, one can obtain similar formulas for $k_{\mathrm{u}}^{\mathrm{wave}}$ by differentiating the formulas above with respect to $t$ and $t'$, as $\dot{F}_t = \partial_t{F_t}$ ($\dot{F}_t*\dot{F}_{t'}$ will only be a generalized function though).

We underline that the proof of Point $(ii)$ in Proposition \ref{prop: F_t * F_tp} makes use of the specificities of the dimension $3$. First in equation \eqref{eq:simplify r2}, where the scalars $r^2$ cancel each other out; second in \eqref{eq:integ_exact} where an exact antiderivative of the integrated function can be computed. None of these two simplifications hold in higher dimension or in dimension 2, and formulas as simple as equation \eqref{f_t star f_t'} are not expected to hold.
\begin{remark}
Expression \eqref{f_t star f_t'} with $h=x-x'$ is the covariance kernel of the solution process $U$ with initial condition the ``formal" white noise process $V^0$ with the stationary Dirac delta covariance kernel $k_{\mathrm{v}}(x,x') = \delta_0(x-x')$: 
\begin{align}
{[(F_t \otimes F_{t'}) * k_{\mathrm{v}}](x,x') = (F_t * F_{t'} * \delta_0)(x-x') = (F_t * F_{t'})(x-x')}.
\end{align}
Somewhat surprisingly, although formula \eqref{f_t star f_t'} yields a summable function over $\mathbb{R}^3$ when $t$ and $t'$ are fixed, it can not be used for practical computations as the diagonal terms of the related covariance matrices are all singularities: $(F_t * F_t)(0) = +\infty$... Yet, formula \eqref{f_t star f_t'} may be used together with explicit kernels $k_S$ to yield usable expressions. For instance, if $k_{\mathrm{v}}(x,x') = k_S(x-x') = C\exp(-{|x-x'|^2}/{2L^2})$, we state without proof that
\begin{align}
    (F_t * F_{t'}&* k_S)(h) = \nonumber\\ \sgn(tt')& \frac{\sqrt{2\pi}}{2}\frac{CL^3}{c^2}\Bigg(\frac{\Phi\big( \frac{R_1 + |h|}{L}\big)-\Phi\big( \frac{R_1 - |h|}{L}\big)}{2|h|}- \frac{\Phi\big( \frac{R_2 + |h|}{L}\big)-\Phi\big( \frac{R_2 - |h|}{L}\big)}{2|h|}\Bigg),
\end{align}
where $h = x-x'$, $\ \Phi(s) = (2\pi)^{-1/2}\int_{-\infty}^s \exp(-t^2/2)dt$, $R_1 = c\big||t|-|t'|\big|$, $R_2 = c(|t| + |t'|) $. Such a kernel always takes finite values: when $h$ goes to $0$, the above formula reduces to well defined derivatives.
\end{remark}

Although these formulas are interesting in their own right, the study of propagation phenomena is usually done thanks to compactly supported initial conditions, which can never be modelled with wide sense stationary random fields. We partially deal with compactly supported initial conditions in Section \ref{sub: radial}, within the context of radial symmetry.
\subsubsection{Radially symmetric initial conditions}\label{sub: radial}
Assume that the sample paths of the process $V^0$ enjoy radial symmetry around some $x_0 \in \mathbb{R}^3$. This can be expressed in terms of differential operators in $(r,\theta,\phi)$, the spherical coordinate system  around $x_0$:
\begin{align}\label{eq:as_radial}
    \mathbb{P}(\{\omega \in \Omega: \partial_{\theta}V_{\omega}^0 = 0\}) = 1, \ \ \ \text{ and } \ \ \
    \mathbb{P}(\{\omega \in \Omega: \partial_{\phi}V_{\omega}^0 = 0\}) = 1.
\end{align}
Then by Proposition 3.5 of \cite{hnr_bernoulli}, $k_{\mathrm{v}}$ verifies, in the sense of distributions,
\begin{align}\label{eq:radial_k}
    \forall x \in \mathcal{D},\ \ \ \partial_{\theta}(k_{\mathrm{v}}(x,\cdot)) = 0 \ \ \ \text{ and } \ \ \ \partial_{\phi}(k_{\mathrm{v}}(x,\cdot)) = 0.
\end{align}
Thus, there exists a function $k_{\mathrm{v}}^0$ defined on $\mathbb{R}_+\times\mathbb{R}_+$ such that $k_{\mathrm{v}}(x,x') = k_{\mathrm{v}}^0(r^2,r'^2)$, with $r=|x|$, $r'=|x'|$ (directly using the squares $r^2$ and $r'^2$ will simplify computations later on). Similarly, assume that the sample paths of $U^0$ exhibit radial symmetry and write $k_{\mathrm{u}}(x,x') = k_{\mathrm{u}}^0(r^2,r'^2)$. Because of the generality of Proposition 3.5 from \cite{hnr_bernoulli}, the Gaussianity of $V^0$ and $U^0$ are not required. Furthermore, the same theorem states that equations \eqref{eq:as_radial} and \eqref{eq:radial_k} are in fact equivalent. From the radial representations of $k_{\mathrm{v}}$ and $k_{\mathrm{u}}$, we can deduce the following convolution-free formulas for $k_{\mathrm{v}}^{\mathrm{wave}}$ and $k_{\mathrm{u}}^{\mathrm{wave}}$:
\begin{proposition}\label{prop: radial sym kernel}
Set $K_{\mathrm{v}}(r,r') = \int_0^r \int_0^{r'}k_{\mathrm{v}}^0(s,s')dsds'$. Then for all $z = (x,t) \in \mathbb{R}^3 \times \mathbb{R}$ and $z' = (x',t') \in \mathbb{R}^3 \times \mathbb{R}$,
\begin{align}
k_{\mathrm{v}}^{\mathrm{wave}}(z,z') &= \frac{\sgn(tt')}{16c^2 r r'} \sum_{\varepsilon,\varepsilon' \in \{-1,1\}}\varepsilon \varepsilon' K_{\mathrm{v}}\big( (r + \varepsilon c|t|)^2,(r' + \varepsilon' c|t'|)^2\big), \label{eq:ft ft' gen}\\
k_{\mathrm{u}}^{\mathrm{wave}}(z,z') &= \nonumber \\ \frac{1}{4rr'}&\sum_{\varepsilon,\varepsilon' \in \{-1,1\}}(r+\varepsilon c|t|)(r'+\varepsilon' c|t'|)  \times k_{\mathrm{u}}^0\big((r + \varepsilon c|t|)^2,(r' + \varepsilon' c|t'|)^2\big). \label{eq:ftp ft'p gen}
\end{align}
\end{proposition}
The expressions \eqref{eq:ft ft' gen} and \eqref{eq:ftp ft'p gen} are interesting in that they are much easier and faster to compute than \eqref{eq: kv wave} and \eqref{eq: ku wave}, which require to compute convolutions. 
\subsubsection{Compactly supported initial conditions} Suppose that $v_0$ is compactly supported on a ball $B(x_0,R)$. The Strong Huygens Principle for the 3 dimensional wave equation (\cite{evans1998}, p. 80) states that $F_t *v_0$ is supported on the spherical shell $B(x_0,R + c|t|) \setminus B(x_0,(R - c|t|)_+)$, where $x_+:= \max(0,x)$.
From a GP modelling perspective, assuming that $\text{Supp}(V^0) \subset B(x_0,R)$ amounts to imposing that $V^0(x) = 0 \ a.s.$ if $x \notin B(x_0,R)$. This is equivalent to $\text{Var}(V^0(x)) = k_{\mathrm{v}}(x,x) = 0$ since $V^0$ is assumed centered. The same reasoning in terms of support can be applied to $u_0$ and $U^0$. In the next proposition, we explore the consequences of such compactness assumptions on the radial formulas \eqref{eq:ft ft' gen} and \eqref{eq:ftp ft'p gen}. The new formulas are readily deduced from Proposition \ref{prop: radial sym kernel}, but we state them on their own as they are the ones used in Section \ref{section: num}.
\begin{proposition}\label{prop: radial sym kernel compact}
Let $R_{\mathrm{v}} > 0$ and $R_{\mathrm{u}} > 0$. Let $\alpha \in (0,1)$ and $\varphi_{\alpha}: \mathbb{R}_+ \rightarrow [0,1]$ be a $C^{1}$ decreasing function such that $\varphi_{\alpha}(s) = 1 \ \text{ if } s < \alpha$ and $\varphi_{\alpha}(s) = 0 \ \text{ if } s \geq 1$.
Set the truncated kernels
\begin{align}
k_{\mathrm{v}}^{R_{\mathrm{v}}}(x,x') &= k_{\mathrm{v}}^{0,R_{\mathrm{v}}}(r^2,r'^2) =  k_{\mathrm{v}}^0(r^2,r'^2)\mathbbm{1}_{[0,R_{\mathrm{v}}]}(r)\mathbbm{1}_{[0,R_{\mathrm{v}}]}(r'), \label{eq:truncated kernel v}  \\
k_{\mathrm{u}}^{R_{\mathrm{u}}}(x,x') &= k_{\mathrm{u}}^{0,R_{\mathrm{u}}}(r^2,r'^2) = k_{\mathrm{u}}^0(r^2,r'^2)\varphi\big(r/R_{\mathrm{u}}\big)\varphi\big(r'/R_{\mathrm{u}}\big). \label{eq:truncated kernel u}
\end{align}
Assume now that $V^0 \sim GP(0,k_{\mathrm{v}}^{R_{\mathrm{v}}})$ and $U^0 \sim GP(0,k_{\mathrm{u}}^{R_{\mathrm{u}}})$.
Then, defining the function $K_{\mathrm{v}}(r,r') = \int_0^{r}
\int_0^{r'}k_{\mathrm{v}}^0(s,s')dsds'$, the two following formulas hold
\allowdisplaybreaks
\begin{align}
k_{\mathrm{v}}^{\mathrm{wave}}&(z,z') = \frac{\sgn(tt')}{16c^2 r r'}\times \nonumber \\
 &\sum_{\varepsilon,\varepsilon' \in \{-1,1\}}\varepsilon \varepsilon' K_{\mathrm{v}}\Big(\min\big( (r + \varepsilon c|t|)^2,R_{\mathrm{v}}^2\big),\min\big((r' + \varepsilon' c|t'|)^2,R_{\mathrm{v}}^2\big)\Big), \label{eq:ft ft' compact}\\
k_{\mathrm{u}}^{\mathrm{wave}}&(z,z') = \frac{1}{4rr'} \times \nonumber\\ 
&\sum_{\varepsilon,\varepsilon' \in \{-1,1\}}(r+\varepsilon c|t|)(r'+\varepsilon' c|t'|)k_{\mathrm{u}}^{0,R_{\mathrm{u}}}\big((r + \varepsilon c|t|)^2,(r' + \varepsilon' c|t'|)^2\big). \label{eq:ftp ft'p compact}
\end{align}
\end{proposition}
Notice that the truncated kernels
$k_{\mathrm{v}}^{R_{\mathrm{v}}}$ and $k_{\mathrm{u}}^{R_{\mathrm{u}}}$ are the covariance kernels of the truncated processes $V_{\mathrm{trunc}}^0(x) = \mathbbm{1}_{[0,R_{\mathrm{v}}]}(|x-x_0|)V^0(x)$ and $U_{\mathrm{trunc}}^0(x) = \varphi\big(|x-x_0|/R_{\mathrm{u}}\big)U^0(x)$ respectively. For $k_{\mathrm{u}}^{R_{\mathrm{u}}}$, the truncation procedure has to be sufficiently smooth to compute $(\dot{F}_t*\dot{F}_{t'})*k_{\mathrm{u}}^{R_{\mathrm{u}}}$, which requires to differentiate $k_{\mathrm{u}}^{R_{\mathrm{u}}}$. In contrast, we used a blunt truncation for $k_{\mathrm{v}}^{R_{\mathrm{v}}}$. Strictly speaking, the sample paths of $V_{\mathrm{trunc}}^0(x)$ are not continuous and Proposition \ref{prop: wave kernel} cannot be used on this GP. However, as discussed in \cite{hnr_bernoulli}, Section 4.2.1, it is easily checked that for $V_{\mathrm{trunc}}^0(x)$, all the computations leading to equation \eqref{eq: kv wave} still hold, and thus equation \eqref{eq:ft ft' compact} also holds.

We also observe that such compactly supported kernels can never be stationary as their sample paths are compactly supported.
Using equation \eqref{eq:ft ft' compact}, one can indeed check that $k_{\mathrm{v}}^{\mathrm{wave}}(z,z) = \text{Var}(V(z)) = 0$ as soon as $(r - c|t|)^2 > R_{\mathrm{v}}^2$, ie $ V(z) = 0 \ a.s. $ and likewise for $k_{\mathrm{u}}^{\mathrm{wave}}$: this is the expression of the strong Huygens principle on the kernels $k_{\mathrm{v}}^{\mathrm{wave}}$ and $k_{\mathrm{u}}^{\mathrm{wave}}$.
Such compactly supported kernels may lead to sparse covariance matrices which may then be used for computational speedups (a topic we leave aside in this article).

\subsubsection{Estimation of physical parameters}\label{subsub:estimation}
The wave kernel \eqref{eq:wave kernel}, using for $k_{\mathrm{u}}$ and $k_{\mathrm{v}}$ radially symmetric kernels supported in $B(x_0^{\mathrm{u}},R_{\mathrm{u}})$ and $B(x_0^{\mathrm{v}},R_{\mathrm{v}})$ respectively, has for hyperparameters $\theta = (c,x_0^{\mathrm{u}},R_{\mathrm{u}},\theta_{k_{\mathrm{u}}^0},x_0^{\mathrm{v}},R_{\mathrm{v}},\theta_{k_{\mathrm{v}}^0})$
Among those, $(c,x_0^{\mathrm{u}},R_{\mathrm{u}},x_0^{\mathrm{v}},R_{\mathrm{v}})$ all correspond to physical parameters. Their estimation via likelihood maximisation is numerically investigated in Section \ref{section: num}.
Note that finding the correct radii $R_{\mathrm{u}}$ and $R_{\mathrm{v}}$ is not a well posed problem: if $\text{Supp} (U^0) \subset B(x_0^{\mathrm{u}},R_{\mathrm{u}})$ then $\text{Supp}(U^0) \subset B(x_0^{\mathrm{u}},\alpha R_{\mathrm{u}})$ for any $\alpha \geq 1$ and $\alpha R_{\mathrm{u}}$ is also a suitable candidate for $R_{\mathrm{u}}$. This is discussed in Section \ref{section: num}.

\begin{remark}[GPR, radial symmetry and the 1D wave equation]
It is known that the radially symmetric 3D wave equation is equivalent to the 1D wave equation, by introducing $\Tilde{w}(r,t) = rw(x,t)$, $r=|x|$. 
However, the joint problem of approximating a radially symmetric solution $w$ of Problem \eqref{eq:wave_eq} with GPR \textit{and} searching for the correct source location parameters $(x_0^{\mathrm{u}},R_{\mathrm{u}},x_0^{\mathrm{v}},R_{\mathrm{v}})$ cannot be reduced to the one dimensional case, as the source centers $x_0^{\mathrm{u}}$ and $x_0^{\mathrm{v}}$ both lie in $\mathbb{R}^3$.
\end{remark}

\subsection{The Point Source limit}
The case of the point source deserves a study on its own as it plays a central role for linear PDEs, both in theory \cite{Duffy2015GreensFW} and in applications. For the wave equation, modelling the source term as a point source (i.e. a Dirac mass) is relevant in a number of real life cases: a localized detonation in acoustics, an electric point source in electromagnetics, a point mass in mechanics and so forth. In this section, we will not make use of the Kriging equations \eqref{eq:krig mean} and \eqref{eq:krig cov} as reconstructing an initial condition that is a point source is actually of little interest. Also, reconstructing the wave equation's Green's function thanks to a pointwise approximation such as GPR is expected to yield poor results because this Green's function in particular is not even defined pointwise: it is a family of singular measures, see equation \eqref{eq:ft ftp in 3D}. However, estimating the physical parameters attached to it, essentially the position parameter $x_0$, is a relevant question and an attainable goal. This is the topic of this section, where we study the behaviour of the log marginal likelihood that comes with WIGPR when the initial condition reduces to a point source. On a more general level, this section also serves as an illustration of the very explicit links one may draw between classical PDE based models and Bayesian kernel methods using physics informed kernels. We will restrict ourselves to the case $u_0 = 0$ in equation \eqref{eq:wave_eq} and thus focus on the kernel $k_{\mathrm{v}}^{\mathrm{wave}}(z,z')$. We begin by clarifying the setting in which we will work.
\subsubsection{Setting, assumptions and objectives} 
\begin{enumerate}[label=\textit{(\roman*)},wide,labelwidth=0em,labelindent=0pt]
\item Note $x_1,...,x_q$ the $q$ sensor locations and assume that we have $N$ time measurements in $[0,T]$ corresponding to times $0 = t_1<... < t_N = T$ for each sensor; we have overall $n = Nq$ pointwise observations of a function $w$ that is a solution of the problem \eqref{eq:wave_eq}. The space-time observation locations $(x_i,t_j)$ are stored in a vector $Z = (Z_1|\cdots|Z_q)^T$ where $Z_i:= ((x_i,t_1),...,(x_i,t_N))$ corresponds to the $i^{th}$ sensor. The observations are then stored in the column vector $w_{\mathrm{obs}} = (w(Z_1)|...|w(Z_q))^T$.
\item We assume that the initial condition $v_0$ corresponding to $w$ is almost a point source: in particular it is supported on a small ball $B(x_0^*,R^*)$ where $R^* \ll 1$.
\item We are interested in finding $x_0^*$, the correct source location. To do so, we study the log marginal likelihood associated to the observations $w_{\mathrm{obs}}$, using a covariance kernel associated to initial conditions truncated around a ball $B(x_0,R)$ to be estimated. Set first $k_{x_0}^{\mathrm{R}}(x,x'):= (4\pi R^3/3)^{-2}k_{\mathrm{v}}(x,x')\mathbbm{1}_{B(x_0,R)}(x)\mathbbm{1}_{B(x_0,R)}(x')$ where $k_{\mathrm{v}}$ is a given a covariance function. The pre-factor $(4\pi R^3/3)^{-2}$ is an anticipation of the upcoming Proposition \ref{prop: point source}. We will then use the wave kernel
\begin{align}
k_{x_0}^{\mathrm{wave,R}}((x,t),(x',t')) &= [(F_t \otimes F_{t'})*k_{x_0}^{\mathrm{R}}](x,x'). \label{eq:trunc_point}
\end{align}
We then view $(x_0,R)$ as hyperparameters of $k_{x_0}^{\mathrm{wave,R}}$, and we denote $(x_0^*,R^*)$ the real source position and size.
\item We assume that except for $x_0$, all the other hyperparameters $\theta$ of $k_{x_0}^{\mathrm{wave,R}}$ are fixed. In particular, we assume that $R = R^*$ and $c = c^*$, where $c^*$ is the true celerity parameter appearing in the wave equation.
\end{enumerate}
In that framework, the log-marginal likelihood $p(w_{\mathrm{obs}}|\theta)$ only depends on $x_0$. We thus write $K_{x_0} := k_{x_0}^{\mathrm{wave,R}}(Z,Z)$ and $\mathcal{L}(\theta,\lambda) = \mathcal{L}(x_0,\lambda)$, $\lambda$ being a Tikhonov regularization parameter (see equation \eqref{eq:log lkhd x_0} below). The log-marginal likelihood then writes
\begin{align}\label{eq:log lkhd x_0}
\mathcal{L}(\theta,\lambda) = \mathcal{L}(x_0,\lambda) = w_{\mathrm{obs}}^T(K_{x_0} + \lambda I_n)^{-1}w_ {\mathrm{obs}} + \log \det(K_{x_0} + \lambda I_n).
\end{align}
\subsubsection{Level sets of $\mathcal{L}(x_0,\lambda)$ and GPS localization}
In Figure \ref{fig: point source}, we provide a 3 dimensional image which displays the numerical values of the map $x_0 \mapsto\mathcal{L}(x_0,\lambda)$ that are below a suitable threshold, on a test case. This figure constitutes visual evidence that in the limit $R \rightarrow 0$, recovering a point source location from minimizing the log marginal likelihood provided by the kernel \eqref{eq:trunc_point} reduces to the classic true-angle multilateration method used for example in GPS systems (see e.g. \cite{trilateration_fang}). In this localization method, the user who is located on a sphere (Earth) sends signals to satellites gravitating around the Earth. From the corresponding time measurements, the distance between the satellite and the user is deduced, which in turn defines a sphere (one for each satellite) on which the user is located. The location of the user lies at the intersection of those spheres, and the Earth. At least three satellites are needed for this intersection to be reduced to a point.

On Figure \ref{fig: point source}, three facts in particular are noteworthy; our task will be to explain them mathematically. First, as a function of $x_0$, $\mathcal{L}(x_0,\lambda)$ reaches local minima over the whole surface of spheres centered on each sensor. Second, at the intersection of two of those spheres, the local minima are smaller. Third, the spheres all intersect at a single point $x_0^*$, which is the global minima of $\mathcal{L}(x_0,\lambda)$ and the real source location.
\begin{figure}[!ht]
	\hspace{-20pt}\centerline{\includegraphics[width=1.8\textwidth]{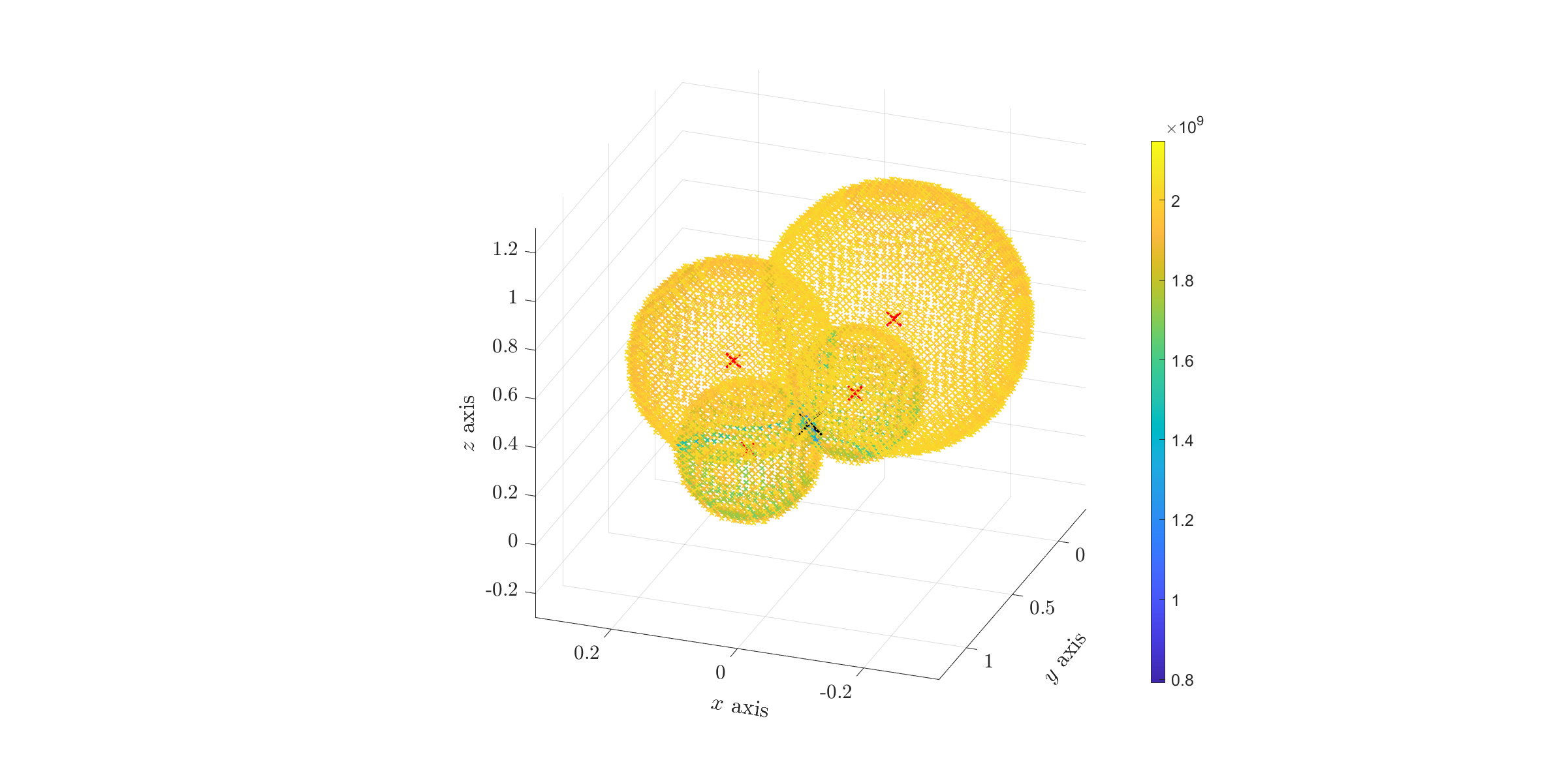}}
	\caption{Negative log marginal likelihood as a function of $x_0 \in \mathbb{R}^3$. Are only represented values of the negative log marginal likelihood that are below $2.035 \times 10^9$. There only remains thin spherical shells. Red crosses: sensor locations. Black cross: source position. The source is located at the intersection of spheres centered at the sensor locations.}
	\label{fig: point source}
\end{figure}

On our way to explaining these three facts, we begin with a convergence statement describing the point source limit, from a covariance point of view. 
\begin{proposition}\label{prop: point source}
Let $k$ be a continuous positive semidefinite function defined on $\mathbb{R}^3 \times \mathbb{R}^3$ and let $x_0 \in \mathbb{R}^3$. For $R > 0$, define $k_{x_0}^{\mathrm{R}}$ its truncation around $x_0$ by
\begin{align*}
    k_{x_0}^{\mathrm{R}}(x,x') =  k(x,x')\mathbbm{1}_{B(x_0,R)}(x)\mathbbm{1}_{B(x_0,R)}(x')/(4\pi R^3/3)^2.
\end{align*} 
Let $t,t' \in \mathbb{R}$. Then $(F_t \otimes F_{t'})*k_{x_0}^{\mathrm{R}}$ defines an absolutely continuous Radon measure over $\mathbb{R}^3 \times \mathbb{R}^3$. Furthermore we have the following weak-$\star$ convergence in the space of Radon measures (i.e. the dual of $C_c(\mathbb{R}^3 \times \mathbb{R}^3)$, the latter space being the space of continuous functions over $\mathbb{R}^3 \times \mathbb{R}^3$ with compact support):
\begin{align}\label{eq:cv point source}
   [(F_t \otimes F_{t'})*k_{x_0}^{\mathrm{R}}] \xrightarrow[R \rightarrow 0]{C_c(\mathbb{R}^3 \times \mathbb{R}^3)'}  k(x_0,x_0) \times (\tau_{x_0} F_t) \otimes (\tau_{x_0}F_{t'}),
\end{align}
where $\tau_{x}\mu$, the translation of $\mu$ by $x$, is defined by $\int f(y)\tau_{x}\mu(dy):= \int f(x+y)\mu(dy)$.
\end{proposition}
As before, the kernel $k_{x_0}^{\mathrm{R}}$ of Proposition  \ref{prop: point source} is the covariance kernel of the truncated process $V_{\mathrm{trunc}}^0(x) = \mathbbm{1}_{B(x_0,R)}(x)V^0(x)/(4\pi R^3/3)$.
The limit object we obtain in equation \eqref{eq:cv point source} is not a function but a singular measure, and thus it cannot be a covariance function. This means that we do not obtain a Gaussian process in the point source limit. More precisely, the Gaussian process associated to the covariance function $k_{x_0}^{\mathrm{wave,R}}$ degenerates into a Gaussian measure \cite{bogachev1998gaussian} over the locally convex space $C_c(\mathbb{R}^3\times\mathbb{R}^3)$ when $R$ goes to zero, though we leave aside this observation for now. On a formal level though, Proposition \ref{prop: point source} provides an entry point for studying the log marginal likelihood \eqref{eq:log lkhd x_0} associated with the kernel \eqref{eq:trunc_point} when $R$ is small. Indeed, Proposition \ref{prop: point source} states that for small values of $R$, the kernel \eqref{eq:trunc_point} behaves like a rank one kernel, i.e. a kernel of the form $k(z,z') = f(z)f(z')$ for some particular function $f$. This observation will prove to be enough for explaining the patterns observed in Figure \ref{fig: point source}.

Properly dealing with the limit $R\rightarrow 0$ implies that we use a mathematical framework compatible with general Radon measures, as indicated by Proposition \ref{prop: point source}. This also implies an additional layer of technicality. Instead, we  introduce regularized (mollified) versions of both the limit object in Proposition \ref{prop: point source} and $\mathcal{L}(x_0,\lambda)$, and study these regularized terms.
This is the content of Propositions \ref{prop: log lik reg} and \ref{prop: log lik reg N}, which are statements on the regularized log marginal likelihood $\mathcal{L}_{\mathrm{reg}}(x_0,\lambda)$ introduced in equation \eqref{eq:log lik reg}. Note however that proving a rigorous mathematical statement linking the behaviours of $\mathcal{L}(x_0,\lambda)$ and $\mathcal{L}_{\mathrm{reg}}(x_0,\lambda)$ is an open question.
\subsubsection{Point source mollification}
We start with regularizing $F_t$ thanks to a mollifier $\varphi(x)$ on $\mathbb{R}^3$ which we choose to be radially symmetric as in \cite{evans2018measure}, section 4.2.1.
Define $\varphi_R(y) = \varphi(y/R)/R^3$, then a $\mathcal{C}^{\infty}_c$ regularization of $F_t$ is obtained by setting $f_t^{\mathrm{R}}(x):= (F_t * \varphi_R)(x)$ for all $x$ in $\mathbb{R}^3$. As $F_t$, $f_t^{\mathrm{R}}$ exhibits radial symmetry. We will next use the following regularizations:
\begin{itemize}[wide,labelwidth=0em,labelindent=0pt]
\item Note $k_{x_0}^{\mathrm{reg}}((x,t),(x',t')):= f_t^{\mathrm{R}}(x-x_0)f_{t'}^{\mathrm{R}}(x'-x_0)$, which plays the role of a regularized version of the limit measure in  Proposition \ref{prop: point source}. The same proposition states that in some sense, when $R$ approaches $0$, $k_{x_0}^{\mathrm{wave,R}}$ is close to $k_{x_0}^{\mathrm{reg}}$. Denote also $F_{x_0}:= ( F_{x_0}^1 | \cdots | F_{x_0}^q)^T$, with $F_{x_0}^i:= (f_{t_1}^{\mathrm{R}}(x_i-x_0),...,f_{t_N}^{\mathrm{R}}(x_i-x_0))$. The covariance matrix corresponding to the hyperparameter $x_0$ is then given by $K_{x_0}^{\mathrm{reg}} = k_{x_0}^{\mathrm{reg}}(Z,Z) = F_{x_0}F_{x_0}^T$. In particular it is rank one.
\item We also assume that $w(x_i,t_j)$ can be approximated by $\tilde{w}(x_i,t_j) = f_{t_j} ^{\mathrm{R}}(x_i-x_0^*)$ as in the point source limit, $v_0 = \delta_{x_0^*}$ and in that case we would have $w(x_i,t_j) = (F_{t_j} * v_0)(x_i) = F_{t_j}(x_i-x_0^*)$ (forgetting for a second that $F_t$ is not defined pointwise). We thus introduce the column vector of ``approximated observations" $W = \big(\tilde{w}(x_i,t_j)\big)_{i,j}$ and we assume that $W$ is ordered as $
W = (W_1 | \cdots | W_q)^T$ where $W_i$ corresponds to the $i^{th}$ sensor: $W_i = (\tilde{w}(x_i,t_1),...,\tilde{w}(x_i,t_N))\in \mathbb{R}^N$.
\end{itemize}
We may then introduce the ``regularized" log marginal likelihood built by replacing $k$ with $k_{x_0}^{\mathrm{reg}}$ and $w_{\mathrm{obs}}$ by $W$:
\begin{align}\label{eq:log lik reg}
\mathcal{L}_{\mathrm{reg}}(x_0,\lambda):= W^T(K_{x_0}^{\mathrm{reg}} + \lambda I_n)^{-1}W + \log \det(K_{x_0}^{\mathrm{reg}} + \lambda I_n),
\end{align}
where we recall that $K_{x_0}^{\mathrm{reg}} = k_{x_0}^{\mathrm{reg}}(Z,Z) = F_{x_0}F_{x_0}^T$. We will then study $\mathcal{L}_{\mathrm{reg}}(x_0,\lambda)$ in the place of $\mathcal{L}(x_0,\lambda)$; as stated before, we expect that $\mathcal{L}(x_0,\lambda)$ behaves similarly to $\mathcal{L}_{\mathrm{reg}}(x_0,\lambda)$, although proofs of such statements are lacking for the moment.
We begin with a proposition which describes the asymptotic behaviour of $\mathcal{L}_{\mathrm{reg}}(x_0,\lambda)$ in the limit of $\lambda \rightarrow 0$. This limit corresponds to noiseless observations, and the limit object in Proposition \ref{prop: log lik reg} provides an explanation of the patterns of Figure \ref{fig: point source}.

\begin{proposition}[Asymptotic behaviour of $\mathcal{L}_{\mathrm{reg}}(x_0,\lambda)$ when $\lambda \rightarrow 0$]\label{prop: log lik reg}
Let $\varepsilon > 0$ and $E_{\varepsilon}:= \{x_0 \in \mathbb{R}^3: ||F_{x_0}||_{\mathbb{R}^n}^2 > \varepsilon \}$. Define the correlation coefficient between $F_{x_0}$ and $W$ by $r(x_0) = \text{Corr}(F_{x_0},W) = \langle F_{x_0},W \rangle_{\mathbb{R}^n} /(||W||_{\mathbb{R}^n} ||F_{x_0}||_{\mathbb{R}^n})$. We set $r(x_0) = 0$ if $F_{x_0} = 0$.
Then we have the following pointwise convergence:
\begin{align*}
\forall x_0 \in \mathbb{R}^3, \ \ \ \big|\lambda \mathcal{L}_{\mathrm{reg}}(x_0,\lambda) - ||W||_{\mathbb{R}^n}^2\big(1-r(x_0)^2\big)\big| = O_{\lambda \rightarrow 0}(\lambda \log \lambda),
\end{align*}
and the uniform convergence on $E_{\varepsilon}$
\begin{align*}
\sup_{x_0 \in E_{\varepsilon}} \big|\lambda \mathcal{L}_{\mathrm{reg}}(x_0,\lambda) - ||W||_{\mathbb{R}^n}^2\big(1-r(x_0)^2\big)\big| = O_{\lambda \rightarrow 0}(\lambda \log \lambda).
\end{align*}
\end{proposition}
The set $E_{\varepsilon}$ is the set of values of $x_0$ for which the vectors $F_{x_0}$ are uniformly large enough for the Euclidean norm. This is interpreted by saying that the elements $x_0$ of $E_{\varepsilon}$ are potential source positions for which the chosen sensor locations should capture a signal with sufficient $L^2$ energy (at least $\varepsilon$ across all sensors) over the window $[0,T]$, should the source be located at $x_0$. Loosely speaking, such locations $x_0$ are ``visible" candidate source positions. From a covariance perspective, we have that $\rho(K_{x_0}^{reg}) = ||F_{x_0}||_{\mathbb{R}^n}^2$, where $\rho$ denotes the spectral radius.
\begin{remark}
In the proof of Proposition \ref{prop: log lik reg}, the determinant term in \eqref{eq:log lik reg} has no influence in the limit object and only pollutes the rate of convergence. Discarding it leads to a $O_{\lambda \rightarrow 0}(\lambda)$ rate of convergence.
\end{remark}

It also makes sense to inspect the case $N \rightarrow \infty$, which is the content of the next proposition; the obtained limit object is similar to that of Proposition \ref{prop: log lik reg}. The limit $N \rightarrow \infty$ corresponds to having the sampling frequency of the sensors go to infinity. In this case, the discrete objects in Proposition \ref{prop: log lik reg} behave as Riemann sums if the time steps $t_k$ are equally spaced and we obtain integrals in the limit  $N \rightarrow \infty$. Notation wise, we highlight the dependence in $N$ in $\mathcal{L}_{\mathrm{reg}}(x_0,\lambda)$ by noting it instead $\mathcal{L}_{\mathrm{reg}}^N(x_0,\lambda)$.
\begin{proposition}[Asymptotic behaviour of $\mathcal{L}_{\mathrm{reg}}^N(x_0,\lambda)$ when $N \rightarrow \infty$]\label{prop: log lik reg N}
Define the following vector valued functions in $L^2([0,T],\mathbb{R}^q)$: 
\begin{align}
\forall t\in [0,T], \hspace{11pt} I_{\mathrm{w}}(t) &:= \big(\tilde{w}(x_1,t),...,\tilde{w}(x_q,t)\big)^T, \nonumber \\
\forall t \in [0,T], \hspace{7pt} I_{x_0}(t) &:= \big(f_t^{\mathrm{R}}(x_1-x_0),...,f_t^{\mathrm{R}}(x_q-x_0)\big)^T. \nonumber
\end{align}
Denote $||\cdot||_{L^2}$ and $\langle,\rangle_{L^2}$ the norm and the dot product of the usual Euclidean structure of $L^2([0,T],\mathbb{R}^q)$.
Assume that the observations are such that $||I_{\mathrm{w}}||_{L^2} > 0$. Introduce then the correlation function, defined whenever $||I_{x_0}||_{L^2} > 0$:
\begin{align}\label{eq:r infinity}
r_{\infty}(x_0) &:= \frac{\langle I_{\mathrm{w}},I_{x_0} \rangle_{L^2} }{||I_{\mathrm{w}}||_{L^2}||I_{x_0}||_{L^2}}.
\end{align}
Assume that for all $k \in \{1,...,N\}$, $t_k = T (k-1)/(N-1)$, i.e. the $t_k$ are equally spaced in $[0,T]$.
Then for all $x_0$ such that $||I_{x_0}||_{L^2} \neq 0$, we have the following pointwise convergence at $x_0$
\begin{align}
\frac{\lambda}{N}\mathcal{L}_{\mathrm{reg}}^N(x_0,\lambda) \xrightarrow[N \rightarrow \infty]{} ||I_{\mathrm{w}}||_{L^2}^2\big(1-r_{\infty}(x_0)^2\big) + q \lambda\log \lambda.
\end{align}
\end{proposition}

\subsubsection{Discussion: location of the point source} 

Propositions \ref{prop: log lik reg} and \ref{prop: log lik reg N} enable us to explain the patterns observed in Figure \ref{fig: point source} where the correct source position is located at the intersection of spheres centered on receivers.
For that purpose, we analyze the limit term in Proposition  \ref{prop: log lik reg} (the same can be done with the one in Proposition \ref{prop: log lik reg N}). We denote $L(x_0)$ the said limit object from Proposition \ref{prop: log lik reg}:
\begin{align*}
L(x_0) = ||W||_{\mathbb{R}^n}^2\big(1-r(x_0)^2\big) = ||W||_{\mathbb{R}^n}^2\Bigg(1 - \frac{\big(\sum_{i=1}^q \langle F_{x_0}^i,W_i\rangle_{\mathbb{R}^n}\big)^2}{||W||_{\mathbb{R}^n}^2 ||F_{x_0}||_{\mathbb{R}^n}^2}\Bigg).
\end{align*}
Note $T_i$ the time of arrival of the point source wave at sensor $i$: $|x_i - x_0^*| = c^*T_i$. Define also $S_i := S(x_i,cT_i)$, the sphere centered on $x_i$, and $A_i$ the thin spherical shell of thickness $2R$ that surrounds $S_i$, given by $A_i := \overline{B(x_0,cT_i + R) \setminus B(x_0,cT_i- R)}$. Then:
\begin{enumerate}[label=\textit{(\roman*)},wide,labelwidth=0em,labelindent=0pt]
\item  $L(x_0)$ reaches a local minima over the whole sphere $S_i$. 
When $x_0$ is located inside $A_i$, the subvectors $W_i$ and $F_{x_0}^i$ of $W$ and $F_{x_0}$ respectively become almost colinear because $f_t^{\mathrm{R}}$ is radially symmetric. They become exactly colinear when $x_0 \in S_i$. This maximizes the term $\langle F_{x_0}^i,W_i\rangle$ in virtue of the Cauchy-Schwarz inequality. When $x_0$ lies in one and only one of those spherical shells $A_i$, the other terms $\langle F_{x_0}^j,W_j \rangle$ are all zero. 
\item The local minima of $L(x_0)$ located at the intersection of two or more spheres $S_i$ are smaller.
More generally, when $I$ is a subset of $\{1,...,q\}$ and when $x_0 \in \bigcap_{i \in I} A_i \setminus \bigcap_{j \notin I} A_j$, the term $\sum_{i \in I} \langle F_{x_0}^i,W_i \rangle$ is (almost) maximized while $\sum_{j \notin I} \langle F_{x_0}^j,W_j \rangle = 0$, which explains why the intersection of spheres are darker coloured than the other parts of the spheres in Figure \ref{fig: point source}.
\item The spheres $S_i$ intersect at a single point, which is exactly $x_0^*$ as well as the global minima of $L(x_0)$. The quantity $r(x_0)$ reaches a global maximum when all subvectors $W_i$ and $F_{x_0}^i$ are colinear, which is the case only when $x_0 \in \bigcap_i S_i$. When there are at least 4 sensors, the intersection of all the spheres $\bigcap_i S_i$ is reduced to at most one point. Recall that we have assumed that $c = c^*$: this implies that $x_0^* \in \bigcap_i S_i$, and thus the minimum of $L(x_0)$ is located at $x_0 = x_0^*$. 
\end{enumerate}

Note that if the speed $c$ in $k_{x_0}^{\mathrm{R}}$ does not correspond to the real speed $c^*$, the intersection $\bigcap_i S_i$ will be empty. Additionally, from an optimization point of view, numerically solving $\inf_{x_0} \mathcal{L}(x_0,\lambda)$ is obviously highly non convex and none of our numerical experiments lead to the correct solution.

\subsection{Initial condition reconstruction and error bounds}\label{sub:inv_pb}
\subsubsection{Initial condition reconstruction procedure} 
Consider a set of space locations $(x_i)_{1\leq i \leq q}$ and moments $(t_j)_{1 \leq j \leq N}$ (imagine $q$ sensors each collecting measurements at time $t_j$ for all $j$). Consider now the following inverse problem: 
\begin{equation}\label{eq:inverse_pb}
\begin{aligned}
\text{Build an approximation of } &u_0 \text{ and } v_0 \text{ from a finite set of measurements} \\ \{w(x_i,t_j)\}_{i,j}
 \text{ where }& (w, u_0, v_0) \text{ are subject to \eqref{eq:wave_eq}}.
\end{aligned}
\end{equation}
We now show that WIGPR provides an answer to the problem \eqref{eq:inverse_pb}. This is not surprising, because the covariance models described in the previous section were derived by putting GP priors over $u_0$ and $v_0$.

As already observed in Section \ref{subsub:gen_formula}, performing GPR on any data with kernel \eqref{eq:wave kernel} automatically produces a prediction $\tilde{m}$ that verifies $\Box \tilde{m} = 0$ in the sense of distributions. Therefore, this function $\tilde{m}$ is the solution of the Cauchy problem \eqref{eq:wave_eq} for some initial conditions $\tilde{u}_0$ and $\Tilde{v}_0$:
\begin{align}\label{eq:def_tildem_wigpr}
    \Tilde{m}(x,t) &= (F_t * \Tilde{v}_0)(x) + (\Dot{F}_t * \tilde{u}_0)(x).
\end{align}
These initial conditions are simply given by $\tilde{u}_0(x) = \tilde{m}(x,0) $ and $\Tilde{v}_0(x) = \partial_t\tilde{m}(x,0)$. If the data $\{w(x_i,t_j)\}_{i,j}$ on which GPR is performed is comprised of observations of a function $w$ that is another solution of problem \eqref{eq:wave_eq}, the initial conditions $(\tilde{u}_0,\Tilde{v}_0)$ can be understood as approximations of the initial conditions $(u_0,v_0)$ corresponding to $w$. More precisely, following Section \ref{subsub:rkhs}, we have $\tilde{m} = p_F(w)$ and thus
\begin{align}
    \tilde{u}_0(x) &= \Tilde{m}(x,0) = p_F(w)(x,0) &&\forall x \in \mathbb{R}^3, \label{eq:U^0 tilde}\\
    \Tilde{v}_0(x) &= \partial_t\Tilde{m}(x,0) = \partial_t p_F(w)(x,0) = p_F(\partial_t w)(x,0)  &&\forall x \in \mathbb{R}^3, \label{eq:V^0 tilde}
\end{align}
where $F$ denotes the finite dimensional space $\text{Span}(k_{\mathrm{w}}(z_1,\cdot),...,k_{\mathrm{w}}(z_n,\cdot))$ and $p_F$ is the orthogonal projector on $F$ with reference to the Hilbert space structure of $H_{k_{\mathrm{w}}}$. Here, $z_m$ is of the form $z_m = (x_i,t_k) \in \mathbb{R}^4$.
This use of WIGPR provides a flexible framework for tackling the problem \eqref{eq:inverse_pb}, as the sensors are not constrained in number or location by any integration formula such as Radon transforms. Taking a look at equations \eqref{eq:U^0 tilde} and \eqref{eq:V^0 tilde}, we can qualitatively discuss the matter of optimal sensor locations for WIGPR. Indeed, we expect that $\tilde{m}$ will provide a better approximation of $w$ when the functions $k_{\mathrm{w}}(z_i,\cdot)_{i=1,...,n}$ are as orthogonal as possible in $\mathcal{H}_{k_{\mathrm{w}}}$, since $\tilde{m}$ is an orthogonal projection on $F$ with reference to the $\mathcal{H}_{k_{\mathrm{w}}}$ inner product. The optimal situation is when given two different sensors $x_i$ and $x_j$, the following should hold for most times $t_k, t_l$:
\begin{align}\label{eq:ortho_F}
\langle k_{\mathrm{w}}((x_i,t_k),\cdot),k_{\mathrm{w}}((x_j,t_l),\cdot) \rangle_{\mathcal{H}_{k_{\mathrm{w}}}} = k_{\mathrm{w}}((x_i,t_k),(x_j,t_l)) \ll 1.
\end{align}
A close inspection of the explicit covariance expressions (equations $(52)$ and $(53)$ from \cite{hnr_bernoulli}) shows that the property \eqref{eq:ortho_F} can be obtained for most times $t_k$ and $t_l$ when the sensors are far apart from each other, as soon as the kernels $k_{\mathrm{u}}$ and $k_{\mathrm{v}}$ are such that $k(x,x') \longrightarrow 0$ when $|x-x'| \longrightarrow +\infty$ (which is common, see e.g. the kernel \eqref{eq:matern 5/2}). Computing optimal sensor locations and obtaining quantitative guaranties of the accuracy of the reconstruction provided by WIGPR is a hard question left for future research.
\subsubsection{Time-dependent error bounds in terms of the initial condition reconstructions} 
Now that we have showed that WIGPR provides approximations of the initial conditions of \eqref{eq:wave_eq}, we underline the fact that these initial condition reconstructions induce a control of the spatial error between the target function $u$ and the Kriging mean $\tilde{m}$, at all times. Indeed, we have the following $L^p$ control in terms of the initial condition reconstruction error. Given $p \in [1,+\infty]$, denote $W^{1,p}(\mathbb{R}^3)$ the Sobolev space of functions $f\in L^p(\mathbb{R}^3)$ whose weak derivatives $\partial_{x_i}f, 1\leq i\leq d$, exist and lie in $L^p(\mathbb{R}^3)$.
\begin{proposition}\label{prop: Lp control}
For any $p \in [1,+\infty]$ and any pair $v_0 \in L^p(\mathbb{R}^3)$, $u_0 \in W^{1,p}(\mathbb{R}^3)$ we have the following $L^p$ estimates for all $t \in \mathbb{R}$:
\begin{align}
    || F_t * v_0 ||_p &\leq |t| \ ||v_0||_p, \label{eq:estim Lp v_0} \\
    || \Dot{F}_t * u_0 ||_p &\leq ||u_0||_p + C_p c|t| \ ||\nabla u_0||_p, \label{eq:estim Lp u_0} 
\end{align}
where $C_p = \Big(\int_S |\gamma|_q^p\ d\Omega/{4\pi}\Big)^{1/p} \leq 3^{1/q} \leq 3, 1/p + 1/q = 1$ $(C_{\infty} = 1, C_1 \leq 1)$. Assume that the correct speed $c$ is known and plugged in $k_{\mathrm{w}}$, equations \eqref{eq:estim Lp v_0} and \eqref{eq:estim Lp u_0} then lead to the following $L^p$ error estimate between the target $w$ and its approximant $\tilde{m}$:
\begin{align}\label{eq:diff w m tilde}
    ||w(\cdot,t) - \Tilde{m}(\cdot,t)||_p \leq |t| \ ||v_0 - \Tilde{v}_0||_p + ||u_0 - \tilde{u}_0||_p + C_pc|t| \ ||\nabla(u_0 - \tilde{u}_0)||_p,
\end{align}
where $\tilde{u_0}$ and $\tilde{v}_0$ are defined in \eqref{eq:U^0 tilde} and \eqref{eq:V^0 tilde}, and $\tilde{m}$ is given in equation \eqref{eq:def_tildem_wigpr}.
\end{proposition}Equations \eqref{eq:estim Lp v_0} and  \eqref{eq:estim Lp u_0} are simple stability estimates for the 3D wave equation, although we have not found them in that form in the literature (notably the explicit control constants $|t|$ and $C_p c|t|$). They fall in the category of Strichartz estimates with $L^p$ control for the space variable and $L^{\infty}$ control for the time variable. We thus provide a proof of Proposition \ref{prop: Lp control}.

Equation \eqref{eq:diff w m tilde} shows that $L^p$ approximations of the initial conditions provide an $L^p$ control between the solution $w$ and the approximation $\tilde{m}$, for any time $t$. This is one reason why in our numerical applications (Section \ref{section: num}), we focus on initial condition reconstruction. 

When $c$ is unknown and estimated by $\hat{c}$ through maximizing the log marginal likelihood, we have instead (highlighting the dependence in $c$ by writing $F_t^c = \sigma_{c|t|}/4\pi c^2 t$)
\begin{align*}
    ||w(\cdot,t) - &\Tilde{m}(\cdot,t)||_p = ||F_t^c * u_0 - F_t^{\hat{c}} * \tilde{u}_0 + \Dot{F}_t^c *v_0 - \Dot{F}_t^{\hat{c}} *\Tilde{v}_0||_p \nonumber \\
    &= ||F_t^c * (u_0 - \tilde{u}_0) + (F_t^c-F_t^{\hat{c}}) * \tilde{u}_0 + \Dot{F}_t^c *(v_0 - \Tilde{v}_0) + (\Dot{F}_t^c- \Dot{F}_t^{\hat{c}}) *\Tilde{v}_0||_p, \nonumber 
\end{align*}    
and thus
\begin{align}
   ||w(\cdot,t) - \Tilde{m}(\cdot,t)||_p \leq &|t| \ ||v_0 - \Tilde{v}_0||_p + ||u_0 - \tilde{u}_0||_p + C_p c|t| \ ||\nabla(u_0 - \tilde{u}_0)||_p\nonumber\\
   & + ||(F_t^c-F_t^{\hat{c}}) * \tilde{u}_0
||_p + ||(\Dot{F}_t^c- \Dot{F}_t^{\hat{c}}) *\Tilde{v}_0||_p.
\end{align}
The terms containing $F_t^c-F_t^{\hat{c}}$ and $\Dot{F}_t^c-\Dot{F}_t^{\hat{c}}$ may be further controlled in terms of $|c - \hat{c}|$ with additional assumptions such as Lipschitz continuity of $u_0$ and $v_0$. Likewise, the quantity $||w(\cdot,t) - \Tilde{m}(\cdot,t)||_p$ may be further controlled if additional assumptions are made on $u_0$ and/or $v_0$. We leave such results to the interested reader. 

\section{Numerical experiments}\label{section: num}
In this section, we showcase WIGPR on functions $w$ that are solutions of Problem \eqref{eq:wave_eq}, using the kernels \eqref{eq:ft ft' compact} and \eqref{eq:ftp ft'p compact} separately as well as together, as in equation \eqref{eq:wave kernel}. The goal is twofold: reconstructing the target function $w$, which more or less amounts to reconstructing its initial conditions (Proposition \ref{prop: Lp control}), and estimating the physical parameters attached. Note that with badly estimated physical parameters, the reconstruction step is more or less bound to fail, especially with inaccurate wave speed $c$ and/or source centers $x_0^{\mathrm{u}}$ and $x_0^{\mathrm{v}}$. 

Running an extensive numerical study of the capabilities and limitations of WIGPR is a large task in itself. For the time being we will settle for simple test cases; in particular we only consider compactly supported and radially symmetric initial conditions, for which we can use the formulas \eqref{eq:ft ft' compact} and \eqref{eq:ftp ft'p compact} which can be evaluated numerically with a low computational cost. We will denote with a star the corresponding true source position $x_0^*$ and celerity $c^*$. whereas their starless counterpart will denote the hyperparameters of the WIGPR kernels. The estimated hyperparameters will be denoted with a hat, e.g. $\hat{c}$.
Two test cases for WIGPR are considered here. A first test case for $k_{\mathrm{u}}^{\mathrm{wave}}$ described in Subsection \ref{sub:GP_pos}, for which $u_0 \neq 0$ and $v_0 = 0$. This would correspond to PAT, though real life PAT test cases would be very unlikely to enjoy radial symmetry properties. A second test case for $k_{\mathrm{u}}^{\mathrm{wave}} + k_{\mathrm{v}}^{\mathrm{wave}}$ described in Subsection \ref{sub:GP_mix}, for which $u_0 \neq 0$ and $v_0 \neq 0$.
For each test case, the full procedure described below is performed.

\paragraph{Numerical simulation and database generation} Given initial conditions $u_0$ and $v_0$, we numerically simulate the solution $w$ over a given time period. We use a basic two step explicit finite difference time domain (FDTD) numerical scheme for the wave equation as described in \cite{bilbao2004}, equation A.24, over the cube $[0,1]^3$.
We also use first order Engquist-Majda transparent boundary conditions \cite{Engquist1977AbsorbingBC}, in order to mimic a full space simulation. We use a sample rate $SR = 200 \ Hz$ (time step $\Delta t = 1/200$ s), a space step $\Delta x = 43 \ mm$, and a wave speed $c^* = 0.5 \ m/s$. The simulation duration is $T = 1.5 \ s$.

30 sensors are scattered in the cube $[0.2,0.8]^3$ using a Latin hypercube repartition and a minimax space filling algorithm. Signal outputs correspond to time series for each sensor, with a sample rate of $50 \ Hz$, so $75$ data points altogether spanned over the time interval $[0, T]$ for each sensor. This leads to $30 \times 75 = 2250$ observations. Each signal is then polluted by a centered Gaussian white noise with standard deviation $\sigma_{\mathrm{noise}} = 0.45$ (resp. $0.09$) for the test case \#1 (resp. test case \#2). These values correspond to around $10\%$ of the maximal amplitude of the signals, see Figures \ref{fig_cos pos sig} and \ref{fig_mix sig}.

\paragraph{Perform WIGPR on simulated data} We perform WIGPR on portions of the dataset obtained above, using the \texttt{kergp} package \cite{citekergp} from R \cite{citeR}. For that we use kernels \eqref{eq:ft ft' compact} and/or \eqref{eq:ftp ft'p compact} which are ``fast" to evaluate, with $K_{\mathrm{v}}$ and $k_{\mathrm{u}}^0$ both 1D $5/2-$Matérn kernels. This Matérn kernel is stationary and writes, in term of the increment $h = x-x'$,
\begin{align}\label{eq:matern 5/2}
k_{5/2}(h) = \sigma^2 \big(1 + {|h|}/{\rho} + {|h|^2}/{3\rho^2} \big)\exp\big(-{|h|}/{\rho} \big).
\end{align}
It has two hyperparameters on its own, $\rho$ and $\sigma^2$. $\rho$ is the length scale of the kernel \eqref{eq:matern 5/2} and should correspond to the typical variation length scale of the function approximated with GPR; $\sigma^2$ is the variance of the kernel. We tackle two different questions related to WIGPR which are respectively the estimation of physical parameters and the sensitivity to sensor locations.
\begin{enumerate}[label=(\subscript{P}{{\arabic*}}),wide,labelwidth=0em,labelindent=0pt]
\item \label{num:multistart} We first study how well the physical parameters $(c^*,x_0^*,R^*)$ can be estimated with WIGPR. For this, we first select $N_s$ time series corresponding to the first $N_s$ sensors with $N_s \in \{3,5,10,15,20,25,30\}$. The corresponding Kriging database contains $75 \times N_s$ data points. For this database, we perform negative log marginal likelihood minimization to estimate the corresponding hyperparameters, which are
\begin{align}\label{eq:def_theta}
\theta = 
\begin{cases}
(x_0^{\mathrm{u}},R_{\mathrm{u}},\theta_{k_{\mathrm{u}}^0},c,\lambda) \in \mathbb{R}^{8} &\text{ if } v_0 = 0 \text{ and } u_0 \neq 0,\\
(x_0^{\mathrm{u}},R_{\mathrm{u}},\theta_{k_{\mathrm{u}}^0},x_0^{\mathrm{v}},R_v,\theta_{k_{\mathrm{v}}^0},c,\lambda) \in \mathbb{R}^{14} &\text{ if } v_0 \neq 0 \text{ and } u_0 \neq 0.
\end{cases}
\end{align}
$\lambda$ corresponds to $\sigma^2$ in Section \ref{subsub:tuning_cov}, and is viewed as an additional hyperparameter in the log marginal likelihood. We use a COBYLA optimization algorithm to optimize $\mathcal{L}(\theta,\lambda)$ and a multistart procedure with $n_{\mathrm{mult}} = 100$ different starting points. That is, $100$ different values of $\theta_0$ are scattered over an hypercube $H \subset \mathbb{R}^8$ or $H \subset \mathbb{R}^{14}$, and the COBYLA log marginal likelihood optimization procedure is run using each value of $\theta_0$ as a starting point. The resulting hyperparameter value providing the minimal negative log marginal likelihood is selected. The multistart procedure mitigates the risk of getting stuck in local maxima. COBYLA is a gradient-free optimization method used in \texttt{kergp} and is available in the \texttt{nloptr} package from R. We then reconstruct the initial conditions using WIGPR, which we evaluate in terms of the indicators in equation \eqref{eq:LP_rel_err}.
\item \label{num:fig_sensibility} Next, we study the sensibility of the reconstruction step with respect to the sensor locations. Consider $40$ different Latin hypercube layouts of the $30$ sensors, each obtained with a minimax space filling algorithm. For each layout, we provide the correct set of hyperparameter values to the model; these values are described in each test case. We then reconstruct the initial conditions using GPR and $N_s$ sensors, with $N_s \in \{3,5,10,15,20,25,30\}$. $L^p$ relative errors (see equation \eqref{eq:LP_rel_err}) are computed between the reconstructed initial condition and the real initial condition. For each number of sensors $N_s$, statistics over the 40 different datasets for these $L^p$ errors are summarized in boxplots (see e.g. Figure \ref{fig_sensib_u_u_L2}). Each box plot shows the median, the first and the third quartiles of a dataset corresponding to results obtained on the 40 different receiver dispositions. The dots inside a circle correspond to the median of each boxplot. The black crosses are the mean of each box plot, which are linked together with the dashed line. The circles are outliers.
\end{enumerate}
In both cases, the approximated initial position $\tilde{u}_0$ is recovered by evaluating the WIGPR Kriging mean at $t = 0$ over a 3D grid and the initial speed $\tilde{v}_0$ is recovered by evaluating the Kriging mean at $t=0$ and $t = \Delta t = 10^{-7}$ over the same 3D grid: $\tilde{v}_0 \simeq (\tilde{m}(\cdot,\Delta t) - \tilde{m}(\cdot,0))/\Delta t$.
Figures are displayed using MATLAB \cite{MATLAB:2020a}.

\paragraph{Numerical indicators} 
For \ref{num:multistart}, we indicate in Tables \ref{table:test_case_1} and \ref{table:test_case_2} the distances between the true physical parameters and the estimated ones, depending on the number of sensors used. Additionally, for every $p \in \{1,2,\infty\}$, we indicate relative $L^p$ reconstruction errors $e_{p,\mathrm{rel}}$ defined below depending on the number of sensors used:
\begin{align}\label{eq:LP_rel_err}
e_{p,\mathrm{rel}}^{\mathrm{u}} = ||u_0-\tilde{u}_0||_p/||u_0||_p\ \ \text{ and } \ \ e_{p,\mathrm{rel}}^{\mathrm{v}} = ||v_0-\tilde{v}_0||_p/||v_0||_p.
\end{align}
A relative error of over $100\%$ means that $||u_0-\tilde{u}_0||_p \geq ||u_0||_p$, in which case the trivial estimator $\hat{u}_0 = 0$ performs better than the estimator $\tilde{u}_0$, in the $L^p$ sense. Note that we deal with three dimensional functions, for which approximation errors are typically larger than for their one dimensional counterpart. Thus, relatively large errors may still correspond to pertinent approximations.
For \ref{num:fig_sensibility} are plotted boxplots of the relative $L^p$ errors over the 40 different sensor layouts, depending on the number of sensors used.
Integrals for the $L^p$ error plots are approximated using Riemann sums over 3D grids containing the support of the integrated functions, with space step $dx = 0.01$.
\begin{figure}[t!]
\begin{subfigure}{\textwidth}
\centering
\includegraphics[width=0.75\textwidth]{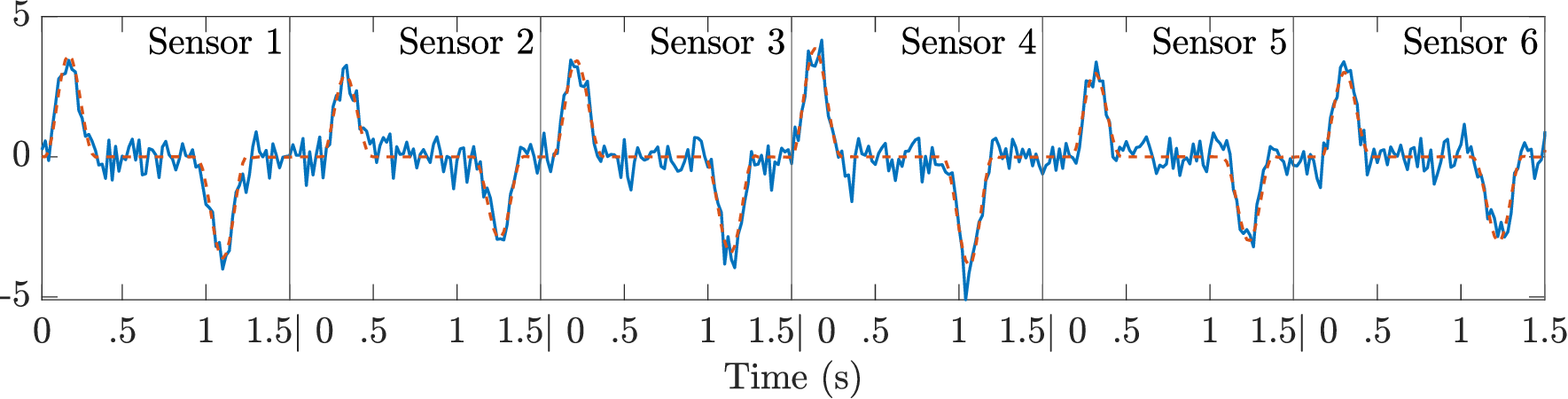}	
\caption{Test case \#1, excerpt of captured signals. Dashed line: noiseless data. Solid line: noisy data.}
\label{fig_cos pos sig}
\end{subfigure}
\vspace{2mm}

\begin{subfigure}{\textwidth}
\centering
\includegraphics[width=0.75\textwidth]{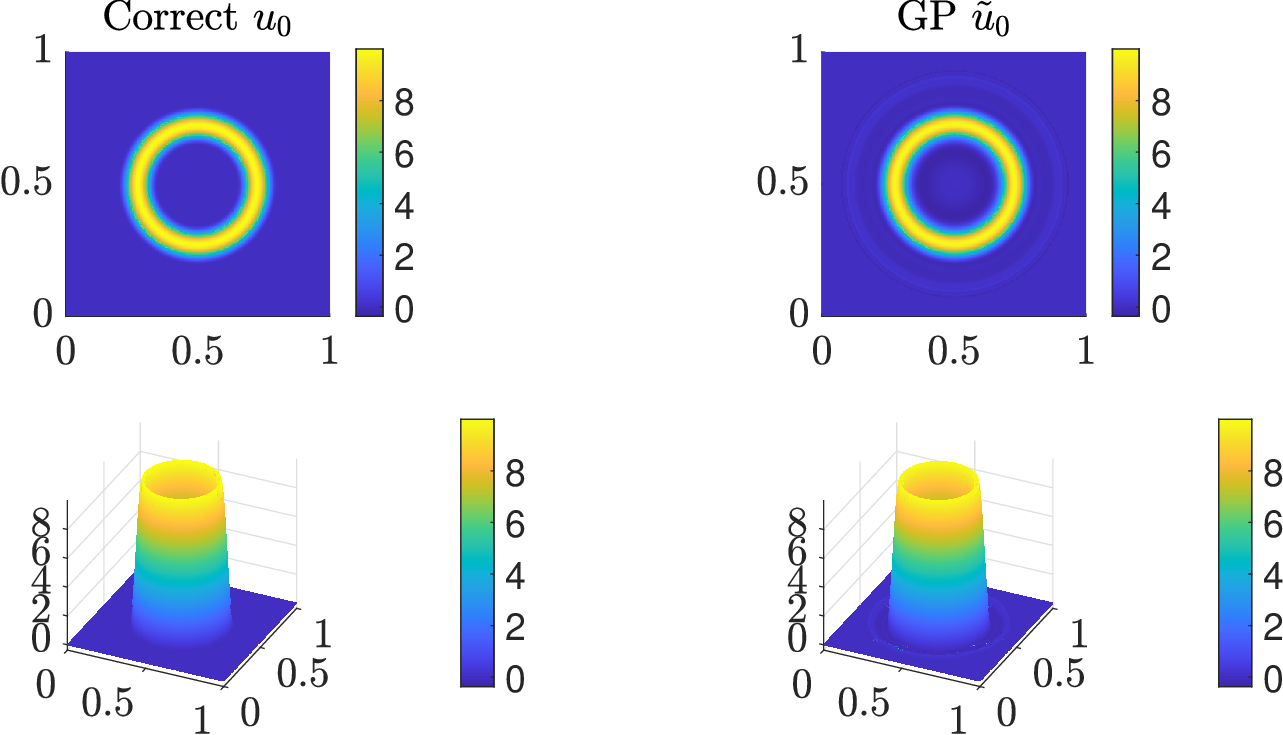}	
\caption{Test case \#1: True $u_0$ (left column) vs WIGPR $u_0$ (right column). 15 sensors were used. The images correspond to the 3D functions evaluated at $z = 0.5$.}
\label{fig_cos pos slice}
\end{subfigure}
\caption{Visualization of signal and WIGPR results for the test case \#1}
\end{figure}

The datasets, the code for generating the datasets and the code for performing WIGPR are available online at \\ 

\centerline{\texttt{https://github.com/iain-pl-henderson/wave-gpr}}

\subsection{Test case for $k_{\mathrm{u}}^{\mathrm{wave}}$}\label{sub:GP_pos}
In this test case, $v_0$ is assumed null and thus we set $k_{\mathrm{v}} = 0$, which yields $k_{\mathrm{v}}^{\mathrm{wave}} = 0$. We thus use $k_{\mathrm{u}}^{\mathrm{wave}}$ defined in \eqref{eq:ftp ft'p compact} for GPR. 
We use the 1D Matérn kernel \eqref{eq:matern 5/2} for $k_{\mathrm{u}}^0$ in equation \eqref{eq:ftp ft'p compact}. The initial condition $u_0$ is a radial ring cosine described as follows. We set $x_0^* = (0.5,0.5,0.5)^T$, $R_1 = 0.15, R_2 = 0.3$ and $A = 5$, the corresponding initial conditions (IC) are given by $v_0(x) = 0$ and
\begin{align*}
    u_0(x) &= A\mathbbm{1}_{[R_1,R_2]}(|x-x_0^*|)\Bigg(1 + \cos\bigg(\frac{2\pi(|x-x_0^*|-\frac{R_1+R_2}{2})}{R_2-R_1}\bigg) \Bigg).
\end{align*}
See Figure \ref{fig_cos pos slice}, left column, for a graphical representation of $u_0$. See Figure \ref{fig_cos pos sig} for an excerpt of the corresponding Kriging database. For problem \ref{num:multistart}, the optimization domain is chosen to be the following hypercube of $\mathbb{R}^8$
\begin{align}\label{eq:dom_optim_1}
\theta &= (x_0,R,\rho,\sigma^2,c,\lambda) \nonumber \\
 &\in [0,1]^3 \times [0.03,0.5] \times [0.02,2] \times [0.1,5]\times [0.2,0.8]\times [10^{-8},1].
\end{align}
For problem \ref{num:fig_sensibility}, the hyperparameter $\theta_0$ provided to the model is
\begin{align}
\theta_0 = (x_0,R,(\rho,\sigma^2),c,\lambda) = ((0.65,0.3,0.5),0.3,(0.2,3),0.5,\sigma_{\mathrm{noise}}^2),
\end{align}
with $\sigma_{\mathrm{noise}}^2 = 0.45^2 = 0.2025$. The value of $0.2$ provided for $\rho$ is a visual estimation of the length scale of $u_0$ based on Figure \ref{fig_cos pos slice}.

\subsubsection{Discussion on the numerical results} For problem \ref{num:multistart}, Table \ref{table:test_case_1} shows that the physical parameters $x_0$ and $c$ are well estimated. The source size parameter $R$ is overestimated, as could be expected from Section \ref{subsub:estimation}. The relative errors show that the overall function reconstruction is overall satisfying, with relative errors below $15\%$ for $N_s = 20, 25$. The noise level $\sigma_{\text{noise}}^2$ (whose estimator is $\hat{\sigma}_{\mathrm{noise}}^2 = \lambda$ in \eqref{eq:def_theta}) is often overestimated. For problem \ref{num:fig_sensibility} (figures \ref{fig_sensib_u_u_L2}, \ref{fig_sensib_u_u_Linf} and \ref{fig_sensib_u_u_L1}), the relative errors stagnate below $10\%$. The IQR (interquartile range, i.e. the difference between the $3^{rd}$ and the $1^{st}$ quartiles) remains below $2\%$. This means that for this test case, the reconstruction step is not very sensitive to the sensors layout when they are scattered as a Latin hypercube.
\vspace{1cm}

\begin{table}[h!]
\centering
\small
\begin{tabular}{|c|c|c|c|c|c|c|c|c|} \hline
$N_{\mathrm{sensors}}$ & 3     & 5     & 10    & 15    & 20    & 25    & 30 & Target   \\ \hline
$|\hat{x_0}-x_0^*|$      & 0.204 & 0.003 & 0.004 & 0.008 & 0.003 & 0.004 & 0.015 & 0 \\
$\hat{R_{\mathrm{u}}}$           & 0.386 & 0.432 & 0.462 & 0.431 & 0.414 & 0.471 & 0.452 & 0.25 \\
$|\hat{c}-c^*|$       & 0.084 & 0.004 & 0.005 & 0.005 & 0.006 & 0.001 & 0.004 & 0\\
$\hat{\sigma}_{\mathrm{noise}}^2$   & 0.917 & 0.879 & 0.93  & 0.99  & 0.361 & 0.988 & 0.377 & 0.2025 \\ \hline
$\hat{\rho}$   & 0.02  & 0.02  & 0.025 & 0.02  & 0.035 & 0.024 & 0.032 & $\sim 0.05$\\
$\hat{\sigma}^2$     & 2.367 & 3.513 & 4.903 & 3.168 & 4.446 & 4.619 & 4.79  & Unknown\\ \hline
$e_{1,\mathrm{rel}}^{\mathrm{u}}$         & 1.275 & 0.157 & 0.128 & 0.168 & 0.11  & 0.103 & 0.248 & 0\\
$e_{2,\mathrm{rel}}^{\mathrm{u}}$         & 1.056 & 0.095 & 0.082 & 0.124 & 0.088 & 0.064 & 0.213 & 0\\
$e_{\infty,\mathrm{rel}}^{\mathrm{u}}$        & 1.037 & 0.132 & 0.128 & 0.198 & 0.136 & 0.101 & 0.321 & 0 \\ \hline
\end{tabular}
\caption{Hyperparameter estimation and relative errors, test case \#1}
\label{table:test_case_1}
\end{table}
\begin{figure}[h!]
\begin{subfigure}[b]{0.45\textwidth}
\includegraphics[width=\textwidth,height=3.95cm,keepaspectratio]{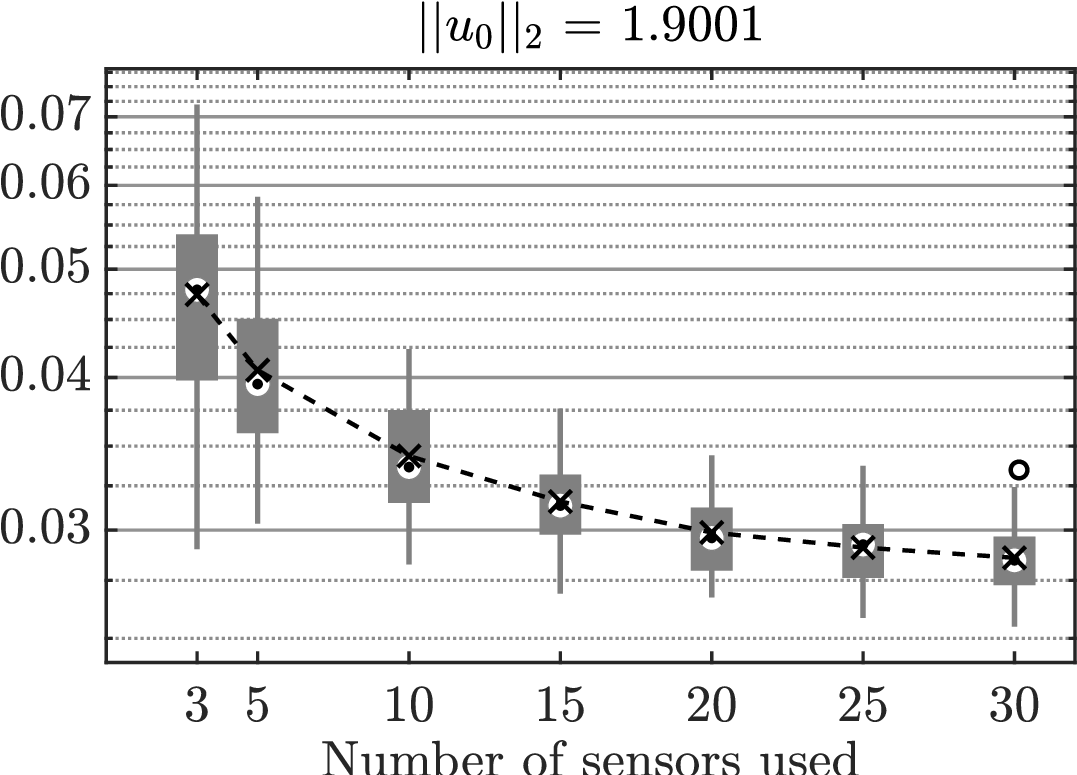}	
\caption{$L^2$ rel. error for $u_0$ (test case \#1)}
\label{fig_sensib_u_u_L2}
\end{subfigure}
\hfill
\begin{subfigure}[b]{0.45\textwidth}
\includegraphics[width=\textwidth,height=3.95cm,keepaspectratio]{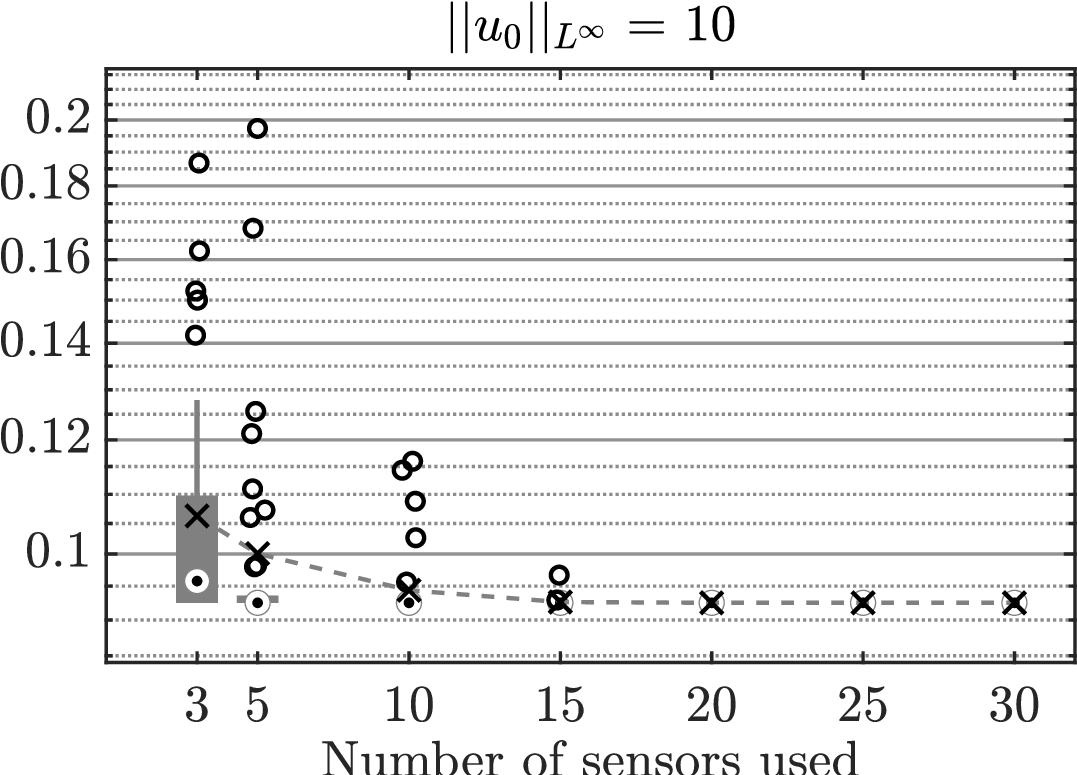}	
\caption{$L^{\infty}$ rel. error for $u_0$ (test case \#1)}
\label{fig_sensib_u_u_Linf}
\end{subfigure}
\centering
\begin{subfigure}[b]{0.45\textwidth}
\includegraphics[width=\textwidth,height=3.95cm,keepaspectratio]{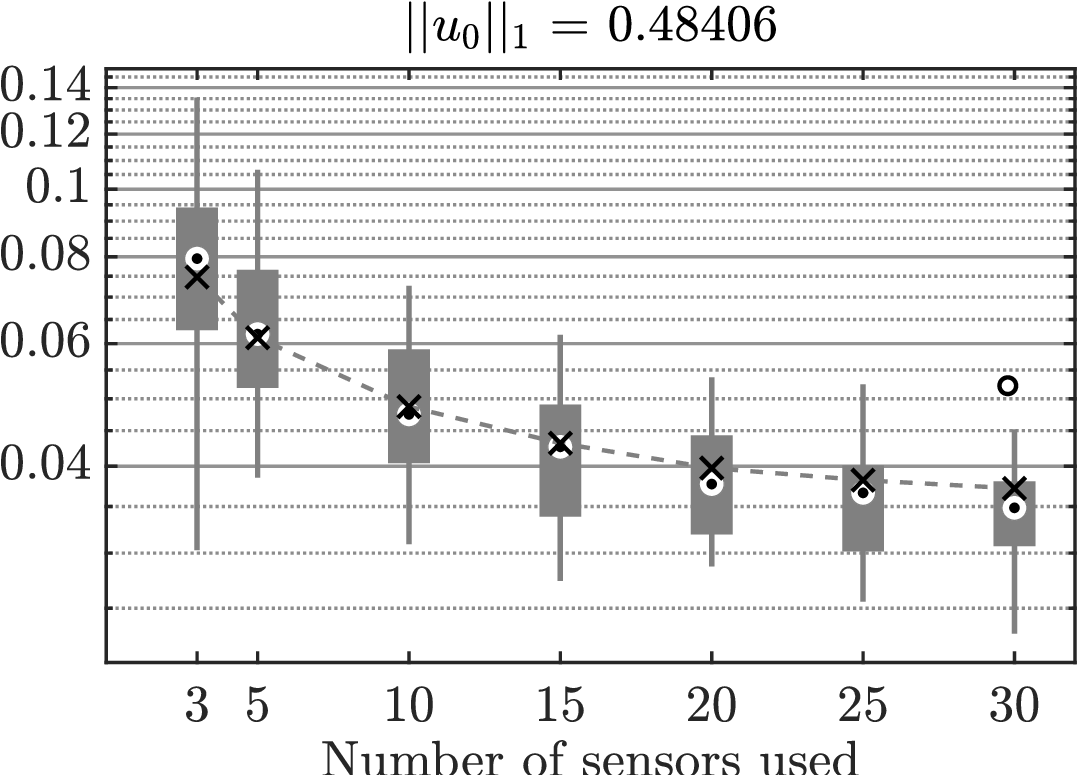}	
\caption{$L^{1}$ rel. error for $u_0$ (test case \#1)}
\label{fig_sensib_u_u_L1}
\end{subfigure}
\caption{Box plots for the sensibility analysis, test case \#1}
\label{sensib_1_and_2}
\end{figure}

\subsection{Test case for $k_{\mathrm{u}}^{\mathrm{wave}} + k_{\mathrm{v}}^{\mathrm{wave}}$}\label{sub:GP_mix}
For this test case, the initial position is a raised cosine, while the initial speed is a ring cosine. We set $x_0^{\mathrm{u}*} = (0.65,0.3,0.5)^T$, $R_u = 0.25$, $A_{\mathrm{u}} = 2.5$, $x_0^{\mathrm{v}*} = (0.3,0.6,0.7)^T$, $R_1^{\mathrm{v}} = 0.05$, $R_2^{\mathrm{v}} = 0.15$ and $A_{\mathrm{v}} = 30$. The corresponding IC are given by
    \begin{align*}
    \begin{cases}
    u_0(x) &= A_{\mathrm{u}}\mathbbm{1}_{[0,R_u]}(|x-x_0^{\mathrm{u}*}|)\Bigg(1 + \cos\bigg(\frac{\pi|x-x_0^{\mathrm{u}*}|}{R_u}\bigg) \Bigg), \\
    v_0(x) &= A_{\mathrm{v}}\mathbbm{1}_{[R_1^{\mathrm{v}},R_2^{\mathrm{v}}]}(|x-x_0^{\mathrm{v}*}|)\Bigg(1 + \cos\bigg(\frac{2\pi\big(|x-x_0^{\mathrm{v}*}|-\frac{R_1^{\mathrm{v}}+R_2^{\mathrm{v}}}{2}\big)}{R_2^{\mathrm{v}}-R_1^{\mathrm{v}}}\bigg) \Bigg).
    \end{cases}
\end{align*}
See Figures \ref{fig_mix pos slice} and \ref{fig_mix spd slice}, left columns, for graphical representations of $u_0$ and $v_0$. See Figure \ref{fig_mix sig} for a visualization of the database. For problem \ref{num:multistart}, the optimization domain is chosen to be the following hypercube
\begin{align}\label{eq:dom_optim_3}
\theta = &(x_0^{\mathrm{u}},R_u,(\rho_{\mathrm{u}},\sigma_{\mathrm{u}}^2),x_0^{\mathrm{v}},R_{\mathrm{v}},(\rho_{\mathrm{v}},\sigma_{\mathrm{v}}^2),c,\lambda)\nonumber \\
 \in &[0,1]^3 \times [0.05,0.4]\times [0.02,2] \times [0.1,5] \nonumber \\
 \times &[0,1]^3\times [0.05,0.4]\times [0.02,2] \times [0.1,5]\times [0.2,0.8] \times [10^{-8},2\times 10^{-2}].
\end{align}
For problem \ref{num:fig_sensibility}, the hyperparameter value $\theta_0$ provided to the model is
\begin{align}
\theta_0 = ((0.65,0.3,0.5),0.3,(0.06,3),(0.3,0.6,0.7),0.15,(0.025,3.5),0.5,\sigma_{\mathrm{noise}}^2),
\end{align}
with $\sigma_{\mathrm{noise}}^2 = 0.0081$. The provided values for $(\rho_{\mathrm{u}},\sigma_{\mathrm{u}}^2)$ and $(\rho_{\mathrm{v}},\sigma_{\mathrm{v}}^2)$ are the estimated values from \ref{num:multistart}.

\subsubsection{Discussion of the numerical results} Table \ref{table:test_case_2} shows that  the physical parameters $x_0^{\mathrm{u}}$, $x_0^{\mathrm{v}}$ and $c$ are well estimated. The source radii $R_{\mathrm{u}}$ and $R_{\mathrm{v}}$ are overestimated, as expected from Section \ref{subsub:estimation}. The noise level $\sigma_{\text{noise}}^2$ is generally overestimated. The reconstruction of the initial position $u_0$ yielded satisfactory results with $L^2$ and $L^{\infty}$ relative errors below $25\%$, and an $L^1$ relative error below $35\%$ ($N_s = 10, 15, 20, 25, 30$).
The higher $L^1$ relative error means that the reconstructed function $\tilde{u}_0$ is supported on a larger set than the true function $u_0$, as the $L^1$ norm favours sparsity. For the initial speed $v_0$, the numerical indicators are not as good, reaching minimal values for $N_s = 25$. The corresponding errors for the $L^1$, $L^2$ and $L^{\infty}$ errors are $64\%$, $28\%$ and $64\%$ respectively. Note though that Figure \ref{fig_mix spd slice} (corresponding to $N_s = 20$) shows that WIGPR still managed to capture the ring structure of $v_0$; the corresponding $L^1$ error for $N_s = 20$ is $150\%$ (Table \ref{table:test_case_2}), confirming that the misestimated support radius $R_{\mathrm{v}}$ is heavily penalized by the $L^1$ norm. The reconstruction of $v_0$ for $N_s = 30$ failed (Table \ref{table:test_case_2}).
For problem \ref{num:fig_sensibility}, the numerical indicators are better. For $u_0$, Figures \ref{fig_sensib_mix_L2_u}, \ref{fig_sensib_mix_Linf_u} and \ref{fig_sensib_mix_L1_u}  show that relative error medians stagnate below $5\%$ for $N_s \geq 15$. The corresponding IQR are around $2\%$. For $v_0$ (Figures \ref{fig_sensib_mix_L2_v}, \ref{fig_sensib_mix_Linf_v} and \ref{fig_sensib_mix_L1_v}), the $L^1$, $L^2$ and $L^{\infty}$ relative error medians stagnate at $30\%, 25\%$ and $40\%$ respectively. The corresponding IQR stagnate at $10\%, 5\%$ and $10\%$ respectively.
\vspace{2cm}

\begin{table}[h!]
\centering
\small
\begin{tabular}{|c|c|c|c|c|c|c|c|c|}\hline
$N_{\mathrm{sensors}}$ & 3      & 5      & 10     & 15     & 20     & 25     & 30  & Target \\ \hline  
$|\hat{x}_0^{\mathrm{u}}-{x_0^{\mathrm{u}}}^*|$      & 0.163  & 0.144  & 0.013  & 0.024  & 0.023  & 0.033  & 0.015 & 0   \\
$\hat{R}_{\mathrm{u}}$          & 0.4    & 0.274  & 0.384  & 0.309  & 0.352  & 0.286  & 0.313   &  0.25   \\
$|\hat{x}_0^{\mathrm{v}}-{x_0^{\mathrm{v}}}^*|$       & 0.163  & 0.18   & 0.035  & 0.028  & 0.037  & 0.006  & 0.05   & 0   \\
$\hat{R}_{\mathrm{v}}$          & 0.252  & 0.166  & 0.313  & 0.356  & 0.348  & 0.266  & 0.339  &  0.15   \\ 
$|\hat{c}-c^*|$       & 0.165  & 0.156  & 0.028  & 0.036  & 0.042  & 0.011  & 0.04 &  0  \\
$\hat{\sigma}_{\mathrm{noise}}^2$  & 0.0178 & 0.0184 & 0.0188 & 0.0161 & 0.0187 & 0.0145 & 0.0116  & 0.0081 \\ \hline
$\hat{\rho}_{\mathrm{u}}$     & 0.034  & 0.069  & 0.102  & 0.027  & 0.031  & 0.061  & 0.034 & $\sim 0.05$    \\
$\hat{\sigma}_{\mathrm{u}}^2$      & 4.649  & 4.472  & 4.575  & 2.493  & 0.678  & 3.272  & 2.541  & Unknown  \\
$\hat{\rho}_{\mathrm{v}}$     & 0.057  & 0.027  & 0.044  & 0.053  & 0.085  & 0.022  & 0.012  & $\sim 0.02$   \\
$\hat{\sigma}_{\mathrm{v}}^2$      & 3.91   & 2.538  & 3.05   & 1.545  & 4.886  & 3.575  & 4.346 &  Unknown  \\
\hline
$e_{1,\mathrm{rel}}^{\mathrm{u}}$       & 2.414  & 1.676  & 0.243  & 0.311  & 0.358  & 0.315  & 0.317  &  0   \\
$e_{2,\mathrm{rel}}^{\mathrm{u}}$       & 1.276  & 1.053  & 0.174  & 0.223  & 0.228  & 0.261  & 0.205  &  0   \\
$e_{\infty,\mathrm{rel}}^{\mathrm{u}}$       & 0.732  & 0.608  & 0.136  & 0.174  & 0.231  & 0.212  & 0.228  &  0   \\
$e_{1,\mathrm{rel}}^{\mathrm{v}}$       & 2.865  & 2.796  & 1.315  & 1.42   & 1.51   & 0.645  & 9.784  &  0   \\
$e_{2,\mathrm{rel}}^{\mathrm{v}}$       & 1.492  & 1.812  & 0.694  & 0.616  & 0.736  & 0.284  & 35.75 &  0    \\
$e_{\infty,\mathrm{rel}}^{\mathrm{v}}$       & 1.083  & 1.608  & 0.817  & 0.763  & 0.845  & 0.635  & 2416.682 &  0 \\ \hline
\end{tabular}
\caption{Hyperparameter estimation and relative errors, test case \#2}
\label{table:test_case_2}
\end{table}
\begin{figure}[h!]
\begin{subfigure}[b]{\textwidth}
\centering
\includegraphics[width=0.90\textwidth]{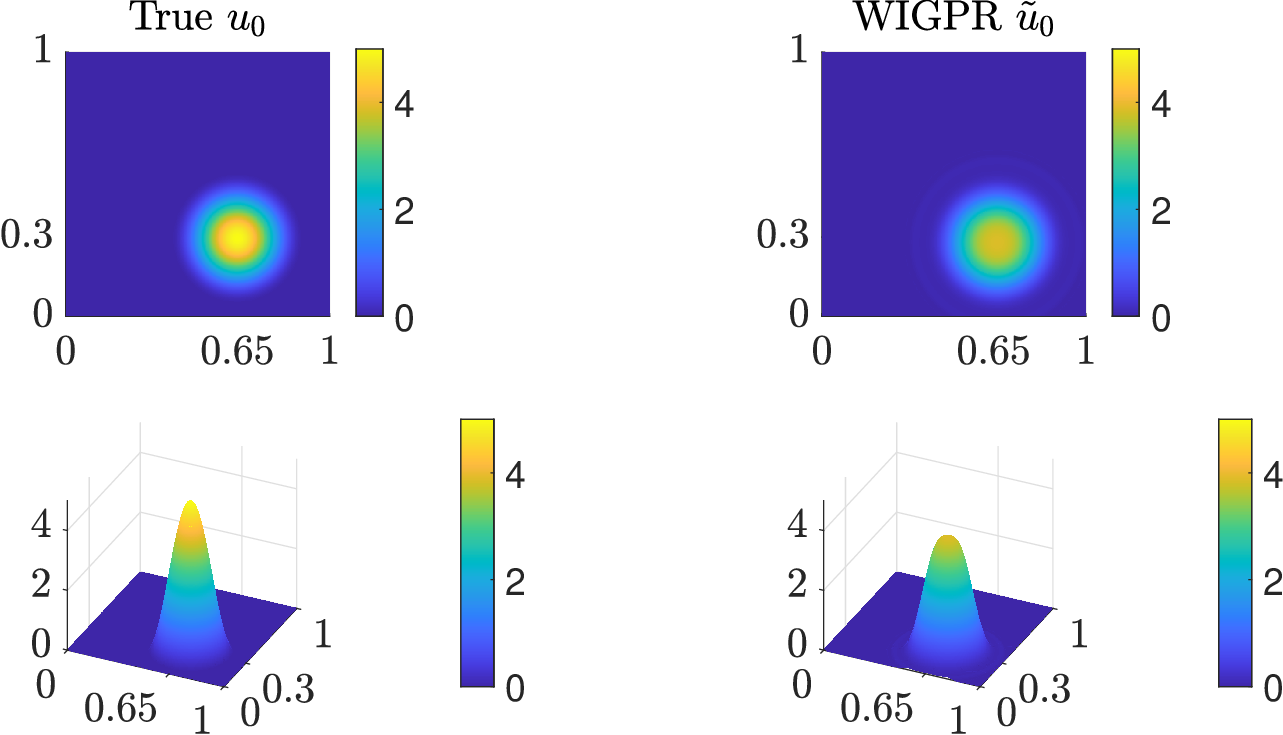}	
\caption{True $u_0$ vs WIGPR $u_0$. The images correspond to the 3D solutions evaluated at $z = 0.5$.}
\label{fig_mix pos slice}
\end{subfigure}
\vspace{0mm}
\begin{subfigure}{\textwidth}
\centering
\includegraphics[width=0.90\textwidth]{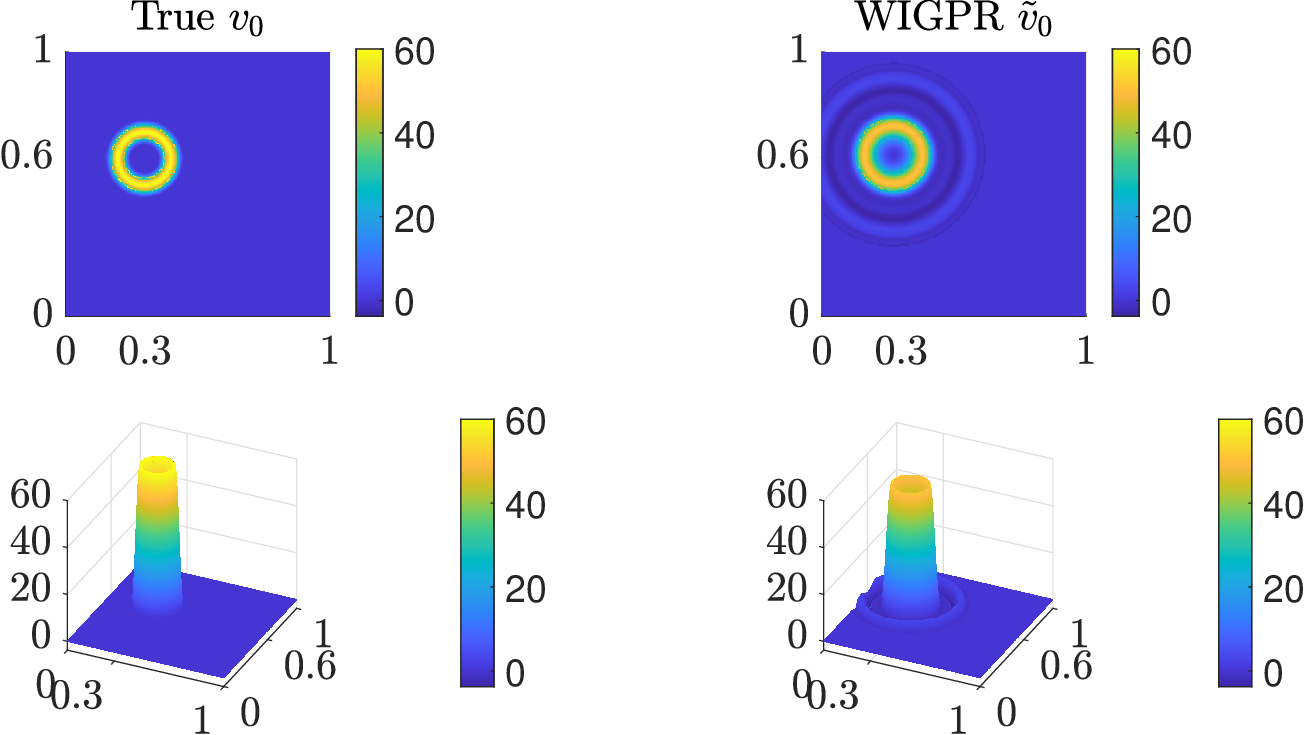}	
\caption{True $v_0$ vs WIGPR $v_0$. The images correspond to the 3D solutions evaluated at $z = 0.7$.}
\label{fig_mix spd slice}
\end{subfigure}
\vspace{0mm}
\caption{Test case \#2: top and lateral view of the reconstructions of $u_0$ (Figure \ref{fig_mix pos slice}) and $v_0$ (Figure \ref{fig_mix spd slice}) provided by WIGPR, in comparison with $u_0$ and $v_0$. Left columns: true IC. Right columns: WIGPR IC reconstructions. 20 sensors were used.}
\label{fig:visu}
\end{figure}

\begin{figure}
\centering
\includegraphics[width=0.7\textwidth]{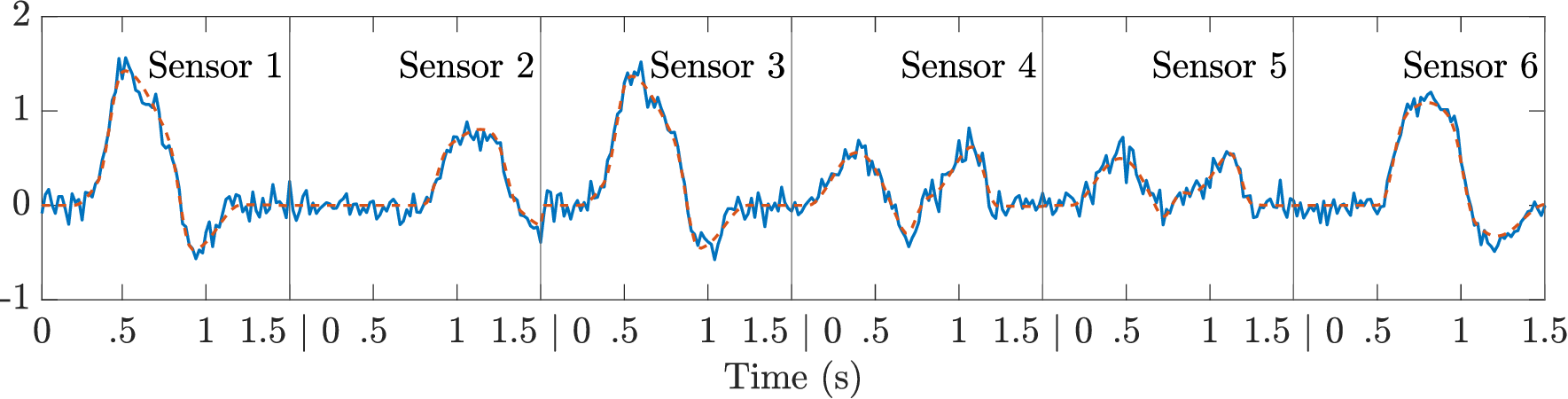}	
\caption{Test case \#2, excerpt of captured signals. Dashed line: noiseless data. Solid line: noisy data.}
\label{fig_mix sig}
\end{figure}

\begin{figure}[h!]

\begin{subfigure}[b]{0.4\textwidth}
\includegraphics[width=\textwidth,height=3.95cm,keepaspectratio]{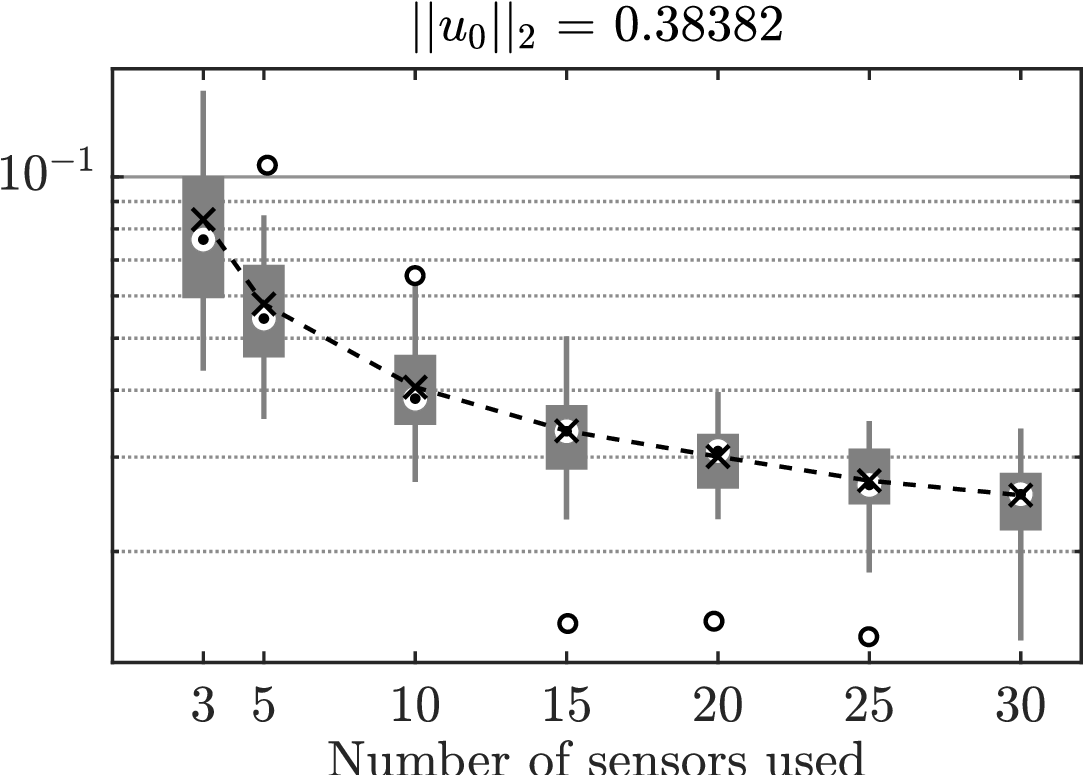}	
\caption{$L^2$ rel. error for $u_0$ (test case \#2)}
\label{fig_sensib_mix_L2_u}
\end{subfigure}
\hfill
\begin{subfigure}[b]{0.4\textwidth}
\includegraphics[width=\textwidth,height=3.95cm,keepaspectratio]{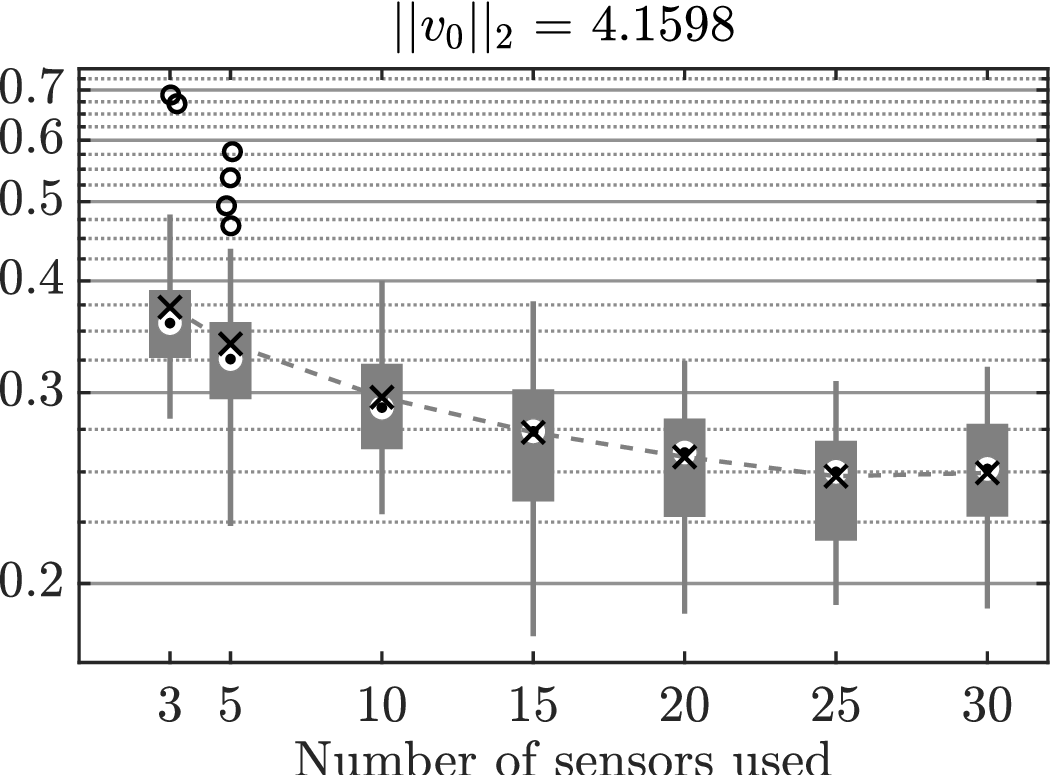}
\caption{$L^2$ rel. error for $v_0$ (test case \#2)}	
\label{fig_sensib_mix_L2_v}
\end{subfigure}\\

\vspace{5mm}

\begin{subfigure}[b]{0.4\textwidth}
\includegraphics[width=\textwidth,height=3.95cm,keepaspectratio]{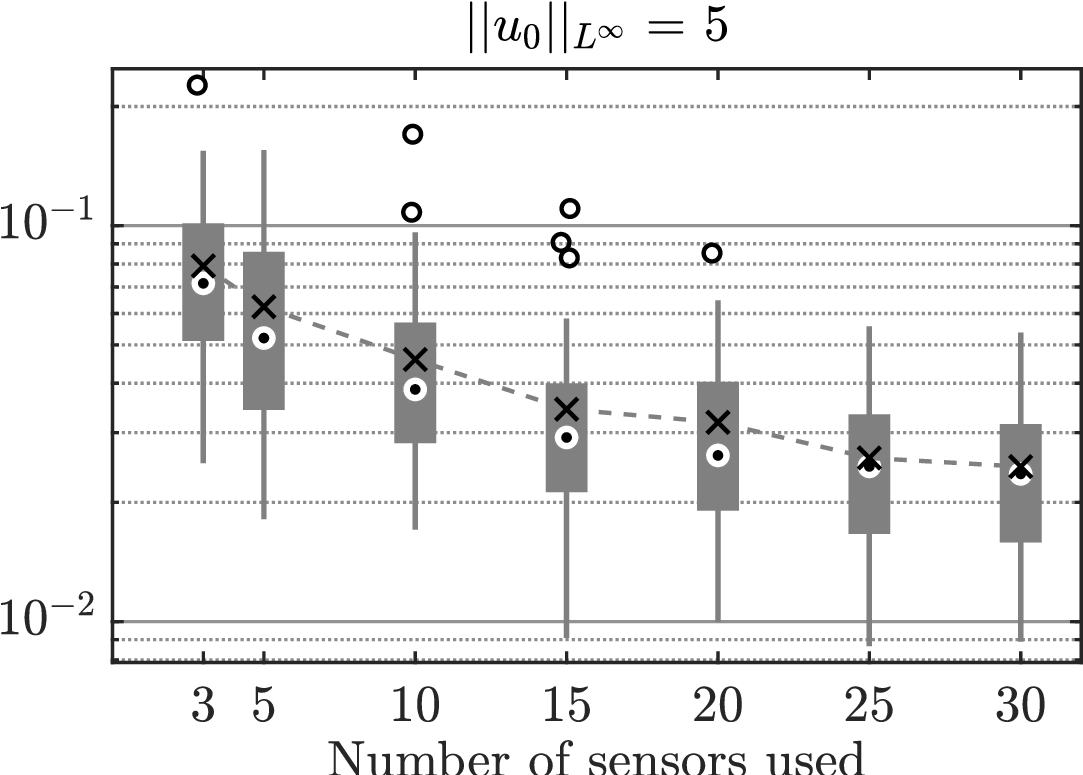}
\caption{$L^{\infty}$ rel. error for $u_0$ (test case \#2)}	
\label{fig_sensib_mix_Linf_u}
\end{subfigure}
\hfill
\begin{subfigure}[b]{0.4\textwidth}
\includegraphics[width=\textwidth,height=3.95cm,keepaspectratio]{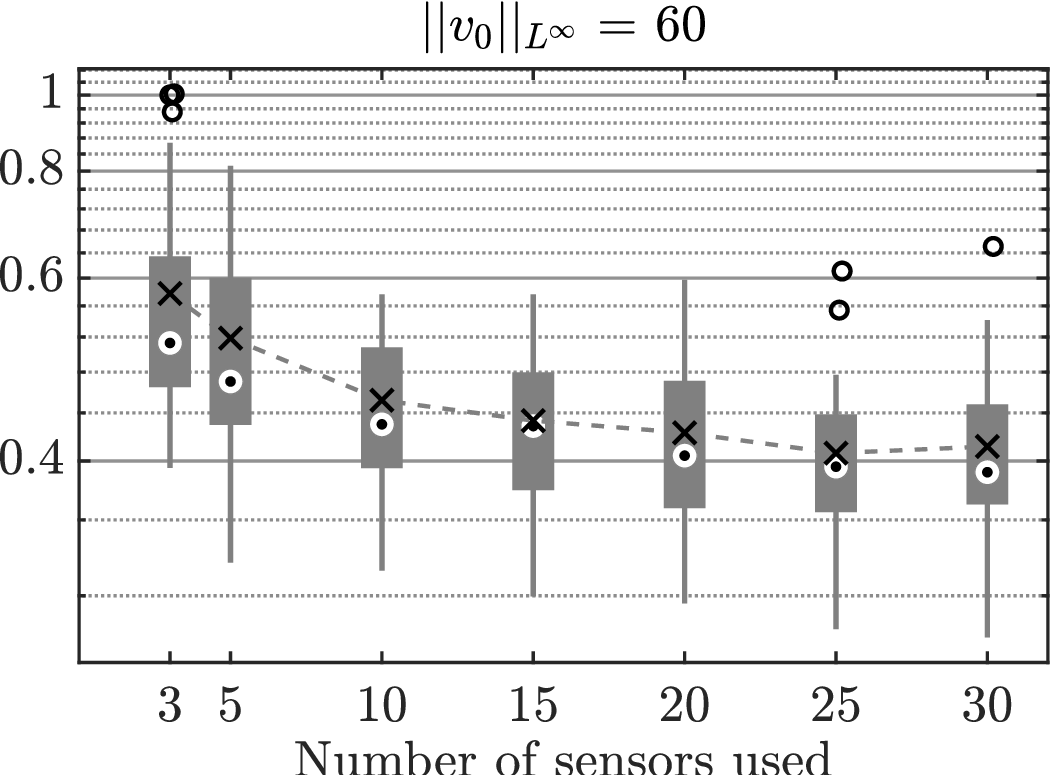}	
\caption{$L^{\infty}$ rel. error for $v_0$ (test case \#2)}	
\label{fig_sensib_mix_Linf_v}
\end{subfigure}\\

\vspace{5mm}

\begin{subfigure}[b]{0.4\textwidth}
\includegraphics[width=\textwidth,height=3.95cm,keepaspectratio]{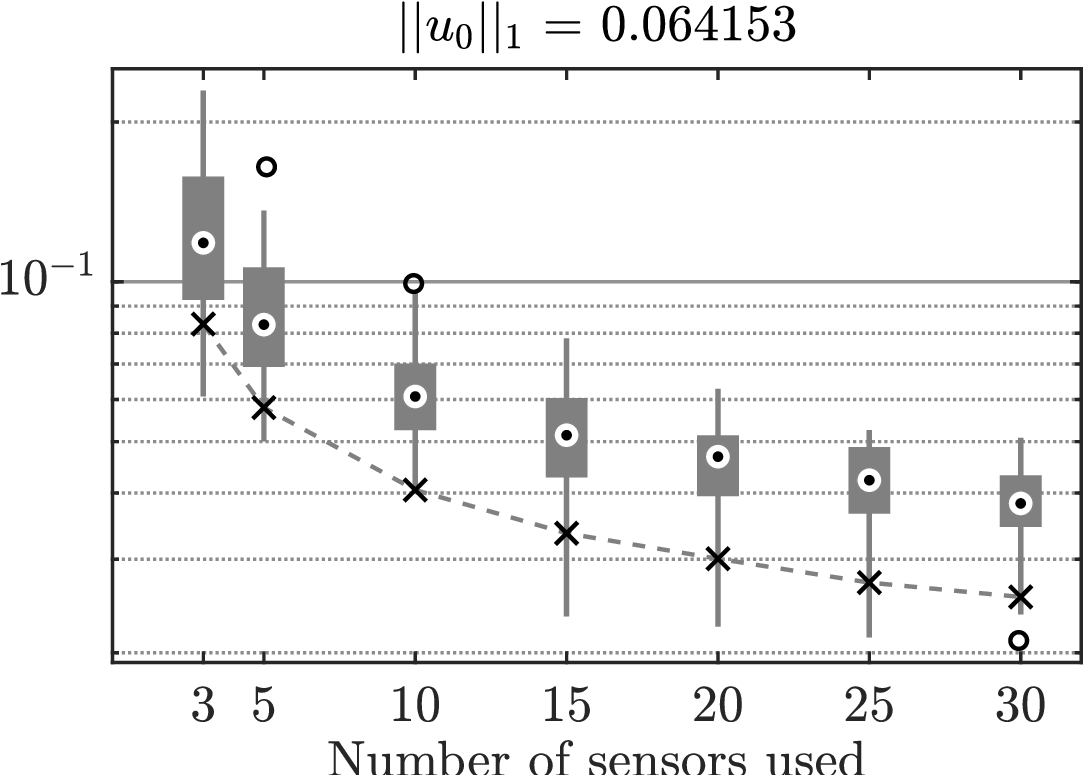}	
\caption{$L^1$ rel. error for $u_0$ (test case \#2)}
\label{fig_sensib_mix_L1_u}
\end{subfigure}
\hfill
\begin{subfigure}[b]{0.4\textwidth}
\includegraphics[width=\textwidth,height=3.95cm,keepaspectratio]{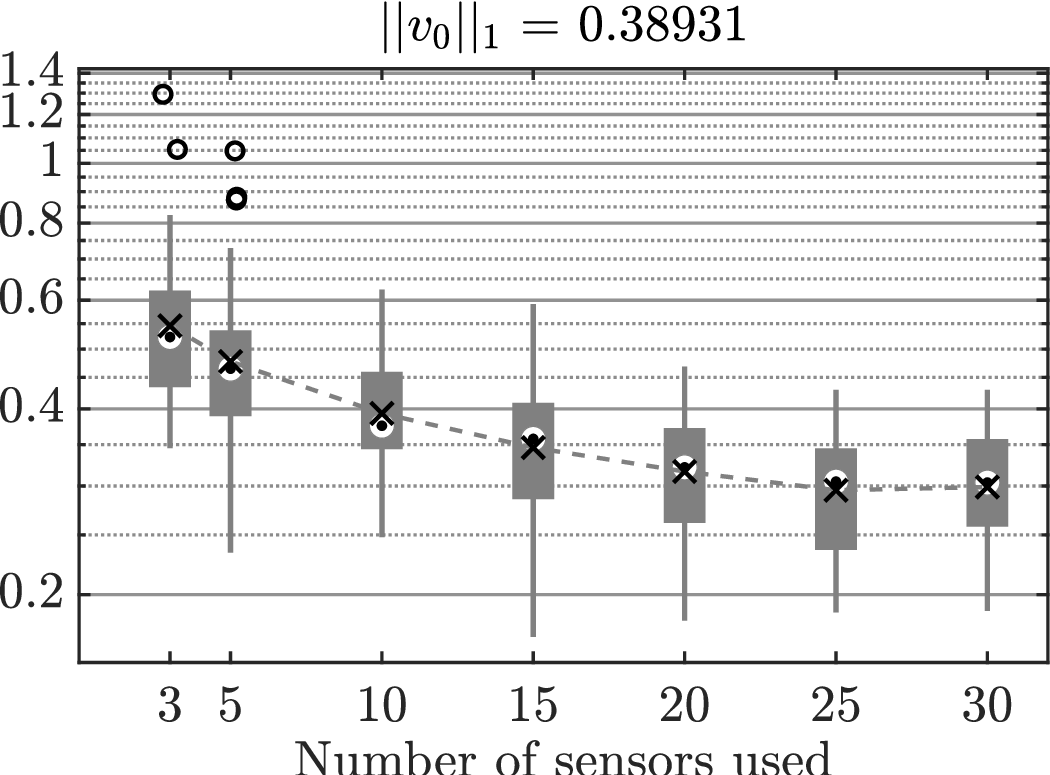}	
\caption{$L^1$ rel. error for $v_0$ (test case \#2)}
\label{fig_sensib_mix_L1_v}
\end{subfigure}

\caption{Box plots for the sensibility analysis, test case\#2}
\label{sensib_3}
\end{figure}

\clearpage

\section{Conclusion and perspectives}

In Section \ref{section:GP_wave}, we described several covariance models tailored to the wave equation; they are particular cases of general ones first derived in a previous work. They correspond to the cases where either wide sense stationarity or radial symmetry assumptions over the initial conditions hold. In addition, the sample paths of the associated random fields (not necessarily Gaussian) are a.s. solution to the homogeneous wave equation. These covariances fully specify centered Gaussian process priors, which can then be used in the context of Gaussian process regression (WIGPR). In that framework, the physical parameter of the PDE system (e.g. source location or wave celerity) can be interpreted as hyperparameters of the WIGPR prior, as in \cite{raissi2017}. 
We then showed that in the limit of the small source radius, the multilateration method for point source localization was naturally recovered by the hyperparameter estimation step of WIGPR.
We furthermore showed that WIGPR naturally provides a reconstruction of the initial conditions of the wave equation, as should be expected when putting probability priors over them. 

The radial symmetry WIGPR formulas from Section 3 were then showcased in Section 4, where two practical questions were tackled. First, WIGPR can correctly estimate certain physical parameters attached to the corresponding wave equation, namely the wave speed and source position. When these parameters are well estimated, WIGPR is capable of providing non trivial reconstructions of the initial condition, which we studied in terms of $L^1$, $L^2$ and $L^{\infty}$ relative errors. We furthermore observed that the reconstruction step was not very sensitive to the layout of the sensors, assuming that the correct set of hyperparameters is provided to the model.

Future possible investigations concern the practical use of the more general formula \eqref{eq:wave kernel} without any radial symmetry assumptions, e.g. for PAT applications. To compute the convolutions efficiently, one may then resort to multidimensional fast Fourier transforms. Moreover, in this first study, we have only used simple methods for GP numerical evaluation. More advanced GP techniques such as inducing points \cite{rasmussen_inducing} should now be used to handle large datasets such as the ones we have used in Section 4. 
The case of the two dimensional wave equation is also of practical interest, e.g. in oceanography \cite{lannes2D}, and presents many different properties when compared to its 3D counterpart (\cite{evans1998}, p. 80). It would thus deserve a theoretical and practical study in its own right when coupled with GPR. 

Finally, the surprising link drawn between our GPR method and the multilateration localization method suggests that other very explicit links should exist between well-chosen kernel methods and traditional mathematical or numerical methods tailored to given physical models. This is certainly an important direction of research, where GPR stands out as a favourable environment through which the communities of machine learning and mathematical physics may be brought together.

\subsubsection*{Acknowledgements} Research of all the authors was supported by SHOM (Service Hydrographique et Océanographique de la Marine) project ``Machine Learning Methods in Oceanography'' no-20CP07. The authors thank Rémy Baraille in particular for his personal involvement in the project. 

\subsubsection*{Declaration of competing interests} The authors declare that they have no known competing financial interests or personal
relationships that could have appeared to influence the work reported in this paper.

\bibliographystyle{abbrv}
\bibliography{bibliography}

\appendix
\section{Appendix}
\subsection{Convolution and tensor product with measures}\label{sub:conv_measure}
This section follows \cite{hnr_bernoulli}, Section 2.2. Given a measure $\mu$ and a function $f$ over $\mathbb{R}^d$, their convolution $\mu*f$ is the following map (if well-defined):
\begin{align}
    (\mu*f)(x) = \int_{\mathbb{R}^d}f(x-y)\mu(dy).
\end{align}
If $\mu$ is an absolutely continuous measure whose density is some other function $g$ (i.e. $\mu(dy) = g(y)dy$), then $(\mu*f)(x)$ reduces to the usual function convolution $(g*f)(x)$.\\
If $\mu$ and $\nu$ are two measures defined over $\mathcal{D}_1 \subset \mathbb{R}^{d_1}$ and $\mathcal{D}_2  \subset \mathbb{R}^{d_2}$, their tensor product $\mu\otimes\nu$ (i.e. the product measure) is the measure over $\mathcal{D}_1\times\mathcal{D}_2$ characterized by the following property:
\begin{align}\label{def:tensor_prod_mes}
    \int_{\mathcal{D}_1\times\mathcal{D}_2} f(x,y)(\mu\otimes\nu)(dx,dy) = \int_{\mathcal{D}_2}\int_{\mathcal{D}_1} f(x,y)\mu(dx)\nu(dy),
\end{align}
for all continuous and compactly supported function $f$ \footnote{For this characterization to hold, $\mu$ and $\nu$ should be assumed Radon, see \cite{hnr_bernoulli} for further details.}. A more general measure theoretic definition of $\mu\otimes\nu$ exists, but it is really equation \eqref{def:tensor_prod_mes} that we will use.\\
Finally, details on the definition of tensor product and convolution with continuous linear forms over $C^m(\mathcal{D})$ spaces (which are necessary for the abstract definition of $\dot{F}_t*u_0$ and $(\dot{F}_t\otimes\dot{F}_{t'})*k_{\mathrm{u}}$) are given in \cite{hnr_bernoulli}, Section 2.2.
\subsection{Proofs}
\begin{proof}[Proof of Proposition \ref{prop: F_t * F_tp}]
$(i):$ Assume for simplicity that $c=1$. Using the definition of the convolution against the measure $F_t \otimes F_{t'}$ (see e.g. \cite{treves2006topological}, Exercise 26.1 p. 282),
\begin{align*}
[(F_t \otimes F_{t'}) * k_{\mathrm{v}}](x,x')&= 
\int_{\mathbb{R}^3 \times \mathbb{R}^3} k(x-s_1,x'-s_2)dF_t(s_1) dF_{t'}(s_2)\\
&= \int_{\mathbb{R}^3 \times \mathbb{R}^3} k_S(x-x'-s_1+s_2)dF_t(s_1) dF_{t'}(s_2).
\end{align*}
But $S$ is invariant under the change of variable $S \ni \gamma \longmapsto - \gamma$ and thus for any continuous function $f$, $\int_{\mathbb{R}^3} f(s_2)dF_{t'}(s_2) = \int_{\mathbb{R}^3} f(-s_2)dF_{t'}(s_2)$. This yields
\begin{align*}
[(F_t \otimes F_{t'}) * k_{\mathrm{v}}](x,x')= \int_{\mathbb{R}^3 \times \mathbb{R}^3} k_S(x-x'-s_1-s_2)dF_t(s_1) dF_{t'}(s_2).
\end{align*}
Applying the definition of the convolution of measures (see e.g. \cite{bogachev1998gaussian}, p. 101) to $F_t * F_{t'}$,
\begin{align*}
(F_t * F_{t'} * k_S)(h) &= \int_{\mathbb{R}^3} k_S(h-s)d(F_t * F_{t'})(s) \\
&= \int_{\mathbb{R}^3}\int_{\mathbb{R}^3} k_S(h-s_1-s_2)dF_t(s_1)dF_{t'}(s_2).
\end{align*}
Setting $h=x-x'$ finishes the proof of Point $(i)$. \\
$(ii):$ Without loss of generality we assume that $c=1$. The computation is carried out in the Fourier domain. Recall that $F_t$ and $\dot{F}_t$ are tempered distributions whose Fourier transforms are given by (\cite{Duistermaat2010}, equation (18.12) p. 294)
\begin{align}\label{eq:fourier of ft ftp}
    \mathcal{F}(F_t)(\xi) = \frac{\sin(ct|\xi|)}{c|\xi|} \ \ \ \text{and} \ \ \ \mathcal{F}(\Dot{F}_t)(\xi) = \cos(ct|\xi|).
\end{align}
We then obtain that (\cite{Duistermaat2010}, Theorem 14.33)
\begin{align}\label{eq:facto fourier}
     \mathcal{F}(F_t * F_{t'})(\xi) = \mathcal{F}(F_t)(\xi)\mathcal{F}(F_{t'})(\xi) = \frac{\sin(t|\xi|)\sin(t'|\xi|)}{|\xi|^2} = \frac{\cos(a|\xi|) - \cos(b|\xi|)}{2|\xi|^2}.
 \end{align}
 with $a = t-t', b = t+t$. We then compute the inverse Fourier transform of the quantity above. Let $h \in \mathbb{R}^3$. In spherical coordinates, noting the unit vectors $\gamma_h = {h}/{|h|}$ and $\gamma = {\xi}/{|\xi|} = {\xi}/{r}$, we define $f_a$ by
 \begin{align}
     f_a(h) &= \int_{\mathbb{R}^3}e^{i \langle h,\xi \rangle} \frac{\cos(a|\xi|)}{|\xi|^2}d\xi = 
    \int_0^{+\infty}\int_0^{2\pi} \int_0^{\pi}e^{i r\langle h,\gamma \rangle} \frac{\cos(ar)}{r^2}r^2\sin \theta d\theta d\phi dr \label{eq:simplify r2}\\
    &=\int_0^{+\infty} \cos(ar) \int_{S}e^{i r|h|\langle \gamma,\gamma_h \rangle}d\Omega dr. \nonumber
 \end{align}
Above, we used the spherical coordinate change $\xi = r\gamma, d\xi = r^2\sin \theta d\theta d\phi dr$. We now make use of radial symmetry in the interior integral, as follow. Note $e_3$ the third vector of the canonical basis of $\mathbb{R}^3$ and $M$ an orthogonal matrix such that $M\gamma_h = e_3$. We perform the change of variable $\gamma' = M\gamma$, using that $MS := \{M\gamma,\gamma\in S\} = S$ and that the corresponding Jacobian is equal to $1$:
 \begin{align}
     \int_{S}e^{i r|h|\langle \gamma,\gamma_h \rangle}d\Omega &= \int_{M S}e^{i r|h|\langle M^T\gamma', \gamma_h \rangle}d\Omega' \label{eq:change M} = \int_{S}e^{i r|h|\langle \gamma', M\gamma_h \rangle}d\Omega'\\
     &= \int_{S}e^{i r|h|\langle \gamma, e_3 \rangle}d\Omega = 2\pi \int_0^\pi e^{i r|h|\cos(\theta)}\sin(\theta)d\theta \nonumber \\
     &= 2\pi \left[-\frac{e^{i r|h|\cos(\theta)}}{ir|h|}\right]_0^\pi = 2\pi \frac{e^{ir|h|}-e^{-ir|h|}}{ir|h|} = 4\pi\frac{\sin(r|h|)}{r|h|}, \label{eq:integ_exact}
 \end{align}
 and thus
 \begin{align}\label{eq:sin a sin b}
     f_a(h) &= 4\pi \int_0^{\infty} \frac{\cos(ar)\sin(|h|r)}{r|h|}dr = 4\pi \int_0^{\infty} \frac{\sin((|h|+a)r) + \sin((|h|-a)r)}{2r|h|}dr \nonumber \\
     &= \frac{2\pi}{|h|}\int_0^{\infty}\frac{\sin(\alpha r)}{r} + \frac{\sin(\beta r)}{r}dr,
 \end{align}
 with $\alpha = |h|+a, \beta = |h|-a$. Finally, we have the Dirichlet integral
 \begin{align}\label{eq:dirichlet}
     \int_0^{\infty}\frac{\sin(\alpha r)}{r}dr = \sgn(\alpha)\frac{\pi}{2}.
 \end{align}
We define the function $f_b$ exactly as $f_a$, and compute it by replacing $a$ by $b$ in every step above. Putting \eqref{eq:facto fourier}, \eqref{eq:sin a sin b} and \eqref{eq:dirichlet} together, the inverse Fourier transform of $\mathcal{F}(F_t*F_{t'})$ is an absolutely continuous measure whose density $f$ is given by
 \begin{align}
     f(h) = \frac{1}{(2\pi)^3}\frac{1}{2}\big(f_a(h) - f_b(h)\big)&= \frac{1}{16\pi |h|}\Big(\sgn(|h| + t-t') + \sgn(|h| - t + t')\nonumber \\& - \sgn(|h| + t + t') - \sgn(|h| -t-t')\Big) \label{ft sign}\\ 
     &=: \frac{1}{16\pi |h|}K(|h|,t,t').\label{eq:def_K}
 \end{align}
$K(|h|,t,t')$ is defined in equation \eqref{eq:def_K}. Note that $K(|h|,-t,t') = -K(|h|,t,t')$ and likewise with $t'$, thus $K(|h|,t,t') = \sgn(t)\sgn(t')K(|h|,T,T')$ with $T = |t|, T' = |t'|$. Using the symmetries in $t$ and $t'$ in equation \eqref{ft sign} and the fact that $\sgn(s) = 1$ if $s > 0$, we obtain 
 \begin{align}
     K(|h|,T,T') &= \hspace{4pt} \sgn(|h| + |T-T'|) + \sgn(|h| - |T-T'|) \nonumber\\
     &\hspace{20pt}- \sgn(|h| + T + T') - \sgn(|h| -T-T') \nonumber\\
     &= 1 + \sgn(|h| - |T-T'|) - 1 - \sgn(|h| -T-T') \nonumber\\
     &= \sgn(|h| - |T-T'|) - \sgn(|h| -T-T'). \label{eq:kt_sign}
 \end{align}
From equation \eqref{eq:kt_sign}, one checks that $K(|h|,T,T') = 0 \text{  if  } |h| < |T-T'|$ or $|h| > T+T'$ and $ K(|h|,T,T') = 2 \text{  if  } |T-T'| < |h| < T + T'$. Thus, $K(|h|,T,T') = 2 \times \mathbbm{1}_{[|T-T'|, T + T']}(|h|)$.
 Identifying the measure $F_t*F_{t'}$ with its density, we obtain
 \begin{align}
     (F_t*F_{t'})(h) = \frac{\sgn(t)\sgn(t')}{8\pi|h|}\mathbbm{1}_{\Big[\big||t|-|t'|\big|, |t| + |t'|\Big]}(|h|),
 \end{align}
 which concludes the proof.
\end{proof}
\begin{proof}[Proof of Proposition \ref{prop: radial sym kernel}]
Without loss of generality, we assume that $c = 1$ and $x_0 = 0$.
We first derive expression \eqref{eq:ft ft' gen}. Let $f$ be a function defined on $\mathbb{R}_+$ and $g$ the function defined on $\mathbb{R}^3$ by $g(x) = f(|x|^2)$. Let $F$ be an antiderivative of $f$ and let $x\in\mathbb{R}^3$.
As in \eqref{eq:change M}, let $M$ be an orthogonal matrix such that $M(x/|x|) = e_3$ and use the change of variable $\gamma' = M\gamma$. As $M S = S$, we have
\begingroup
\allowdisplaybreaks
\begin{align}
    (F_t * g)(x) &= \frac{1}{4\pi t}\int_{S}g(x-t\gamma)t^2d\Omega = \frac{t}{4\pi} \int_{S}f(|x-t\gamma|^2)d\Omega \nonumber\\
    &=  \frac{t}{4\pi} \int_{S}f(|x|^2 + t^2 - 2|t|\langle x,\gamma \rangle )d\Omega \nonumber \\
    &= \frac{t}{4\pi} \int_{M S}f\bigg(|x|^2 + t^2 - 2|t||x|\bigg\langle \frac{x}{|x|},M^T\gamma' \bigg\rangle \bigg)d\Omega' \nonumber \\
    &=  \frac{t}{4\pi} \int_{S}f(|x|^2 + t^2 - 2|t||x|\langle e_3,\gamma \rangle )d\Omega \nonumber \\
    &= \frac{t}{4\pi} \int_{\phi = 0}^{2\pi}\int_{\theta = 0}^{\pi}f(|x|^2 + t^2 - 2|t||x|\cos(\theta))\sin(\theta)d\theta d\phi \nonumber \\
    &= \frac{t}{2} \int_{\theta = 0}^{\pi}f(|x|^2 + t^2 - 2|t||x|\cos(\theta))\sin(\theta)d\theta  \nonumber\\
    &= \frac{t}{4|x||t|}\Big[F(|x|^2 + t^2 - 2|t||x|\cos(\theta))\Big]_{\theta = 0}^{\theta = \pi} \nonumber \\
    &= \frac{\sgn(t)}{4|x|}\Big(F((|x|+|t|)^2) - F((|x|-|t|)^2) \Big)\nonumber \\
   &= \frac{\sgn(t)}{4|x|}\sum_{\varepsilon \in \{-1,1\}}\varepsilon F((|x|+\varepsilon|t|)^2). \label{eq:ft on f}
\end{align}
\endgroup
Introduce now the functions
\begin{align}
    &\tilde{k}(r,r'):= \int_0^{r'} k^0_{\mathrm{v}}(r,s)ds \hspace{5mm} \text{ and } \hspace{5mm} K_{\mathrm{v}}(r,r'):=  \int_0^{r}\int_0^{r'} k^0_{\mathrm{v}}(s,s')ds'ds. \label{def Kint}
\end{align}
We apply twice result \eqref{eq:ft on f} on $k_{\mathrm{v}}$: first by setting $g(x') = k^0_{\mathrm{v}}(|x-t\gamma|^2,|x'|^2)$ where $x-t\gamma$ is fixed, which integrates to $F(s) = \tilde{k}(|x-t\gamma|^2,s)$. Second, by setting $g(x) = \tilde{k}(|x|^2,(|x'|+\varepsilon|t'|)^2)$ where $|x'|+\varepsilon'|t'|$ is fixed, which integrates to $F(s) = K_{\mathrm{v}}(s,(|x'|+\varepsilon'|t'|)^2)$. In detail, we obtain
\begin{align} \label{eq:ft ft proof}
    [(F_t \otimes F_{t'})*k_{\mathrm{v}}](x,x')& = \frac{1}{4\pi t}\frac{1}{4 \pi t'} \int_{S} \int_{S} k^0_{\mathrm{v}}(|x-t\gamma|^2,|x'-t'\gamma'|^2){t'}^2d\Omega' t^2 d\Omega \nonumber\\
    & = \frac{1}{4 \pi t} \frac{\sgn(t')}{4|x'|} \int_{S} \sum_{\varepsilon' \in \{-1,1\}}\varepsilon' \tilde{k}(|x-t\gamma|^2,(|x'|+\varepsilon'|t'|)^2) t^2 d\Omega \nonumber\\
     & = \frac{\sgn(tt')}{16 r r'} \sum_{\varepsilon,\varepsilon' \in \{-1,1\}}\varepsilon \varepsilon' K_{\mathrm{v}}\big( (r + \varepsilon |t|)^2,(r' + \varepsilon' |t'|)^2\big),
\end{align}
which is exactly equation \eqref{eq:ft ft' gen}. By replacing $k_{\mathrm{v}}^0$ with $k_{\mathrm{u}}^0$, we can then use this result to compute 
\begin{align}
[(\Dot{F}_t \otimes \Dot{F}_{t'})*k_{\mathrm{u}}](x,x') = \partial_t \partial_{t'} [(F_t \otimes F_{t'})*k_{\mathrm{u}}](x,x').
\end{align}
First, we compute it for $t \neq 0$ and $t' \neq 0$ by differentiating \eqref{eq:ft ft proof} with reference to $t$ and $t'$, using that for $t \neq 0$, $d|t|/dt = \sgn(t)$ and $d \sgn(t)/dt = 0$. This yields
\begin{align} \label{eq:ftp ftp}
    [(\Dot{F}_t \otimes& \Dot{F}_{t'})*k_{\mathrm{u}}](x,x') \nonumber \\
    &= \frac{1}{4rr'}\sum_{\varepsilon,\varepsilon' \in \{-1,1\}}(r+\varepsilon |t|)(r'+\varepsilon' |t'|)k_{\mathrm{u}}^0\big((r + \varepsilon |t|)^2,(r' + \varepsilon' |t'|)^2\big).
\end{align}
For the case where either $t = 0$ or $t' = 0$, note first from equation \eqref{eq:fourier of ft ftp} that $\mathcal{F}(\dot{F}_0)(\xi) = 1$ and thus $\dot{F}_0 = \delta_0$, the Dirac mass at $0$, which is the neutral element for the convolution. Therefore, when we have both $t = 0$ and $t' = 0$:
\begin{align*}
[(\Dot{F}_0 \otimes \Dot{F}_{0})*k_{\mathrm{u}}](x,x') = [(\delta_0 \otimes \delta_{0})*k_{\mathrm{u}}](x,x') = [\delta_{(0,0)}*k_{\mathrm{u}}](x,x') = k_{\mathrm{u}}(x,x'),
\end{align*}
which is also the result provided by \eqref{eq:ftp ftp} evaluated at $t = t' = 0$. When $t'=0$ and $t \neq 0$, we still have $d|t|/dt = \sgn(t)$ and $d\sgn(t)/dt = 0$, yielding
\begin{align*}
[(\Dot{F}_t \otimes \Dot{F}_{0})*k_{\mathrm{u}}](x,x') &= [(\Dot{F}_t \otimes \delta_{0})*k_{\mathrm{u}}](x,x') = \partial_t [(F_t \otimes \delta_{0})*k_{\mathrm{u}}](x,x') \\
 &= \partial_t \int_{\mathbb{R}^3}\int_{\mathbb{R}^3}k_{\mathrm{u}}^0(|x-y|^2,|x'-y'|^2)F_t(dy)\delta_0(dy') \\
 &= \partial_t \frac{1}{4\pi t}\int_{S} k_{\mathrm{u}}^0(|x-t\gamma|^2,|x'|^2)t^2d\Omega \\
 &= \partial_t \frac{\sgn(t)}{4r} \sum_{\varepsilon \in \{-1,1\}}\varepsilon \tilde{k}((r+\varepsilon |t|)^2,|x'|^2) \\
 &= \frac{1}{2r}\sum_{\varepsilon \in \{-1,1\}}(r+\varepsilon |t|) k_{\mathrm{u}}^0((r+\varepsilon |t|)^2,|x'|^2).
\end{align*}
which is also the result provided by \eqref{eq:ftp ftp} evaluated at $t' = 0$. The same arguments apply to show that expression \eqref{eq:ftp ftp} is valid when $t = 0$ and $t' \neq 0$. Therefore the expression \eqref{eq:ftp ftp} is valid whatever the value of $t,t' \in \mathbb{R}$.
\end{proof}
\begin{proof}[Proof of Proposition \ref{prop: radial sym kernel compact}]
When using the kernel $k_{\mathrm{v}}^{0,R_{\mathrm{v}}}$, we can directly use equation \eqref{eq:ft ft' gen} by substituting $K_{\mathrm{v}}$ with $K_{\mathrm{v}}^{R_{\mathrm{v}}}(r,r'):= \int_0^{r} \int_0^{r'}k_{\mathrm{v}}^{0,R_{\mathrm{v}}}(s,s')dsds'$ and observing that for all $r,r' \geq 0$,
\begin{align*}
K_{\mathrm{v}}^{R_{\mathrm{v}}}(r^2,r'^2):= \int_0^{r^2} \int_0^{r'^2}k_{\mathrm{v}}^{0,R_{\mathrm{v}}}(s,s')dsds' = K_{\mathrm{v}}\Big(\min\big( r^2,R_{\mathrm{v}}^2\big),\min\big(r'^2,R_{\mathrm{v}}^2\big)\Big)
\end{align*}
which directly proves \eqref{eq:ft ft' compact}. Additionally, \eqref{eq:ftp ft'p compact} is only a substitution of $k_{\mathrm{u}}^0$ with $k_{\mathrm{u}}^{0,R_{\mathrm{u}}}$ in \eqref{eq:ftp ft'p gen}: all the mathematical steps are justified as $\varphi \in C^{1}(\mathbb{R}_+)$.
\end{proof}
\begin{proof}[Proof of Proposition \ref{prop: point source}] The proof is carried out by direct computations. First, equation \eqref{eq:kirschoff} yields
\begin{align}\label{eq:piecewise}
[(F_t \otimes F_{t'}) * k_{x_0}^{\mathrm{R}}](x,x') = tt'\int_{S \times S}k_{x_0}^{\mathrm{R}}(x-c|t|\gamma,x'-c|t'|\gamma')\frac{d\Omega d\Omega'}{(4\pi)^2}.
\end{align}
The integrated function in equation \eqref{eq:piecewise} is piecewise continuous over $\mathbb{R}^3 \times \mathbb{R}^3$ and the integral in \eqref{eq:piecewise} is well defined, whatever the values of $x$ and $x'$. Let $f$ be a continuous compactly supported function on $\mathbb{R}^3 \times \mathbb{R}^3$. We define
\begin{align*}
I_R:&= \langle (F_t \otimes F_{t'}) * k_{x_0}^{\mathrm{R}},f \rangle/({4}\pi R^3/3)^2,
\end{align*}
and wish to show that $I_R \rightarrow k(x_0,x_0)\langle \tau_{x_0}F_t \otimes \tau_{x_0}F_{t'},f \rangle$ when $R \rightarrow 0$. Using equation $(52)$ from \cite{hnr_bernoulli} and Fubini's theorem, we have
\begin{align*}
I_R &= \frac{1}{(\frac{4}{3}\pi R^3)^2}\ \int_{\mathbb{R}^3 \times \mathbb{R}^3} f(x,x')[(F_t \otimes F_{t'}) * k_{x_0}^{\mathrm{R}}](x,x')dxdx' \\
&= \frac{1}{(\frac{4}{3}\pi R^3)^2} \int_{\mathbb{R}^3 \times \mathbb{R}^3} f(x,x')tt'\int_{S \times S}k_{x_0}^{\mathrm{R}}(x-c|t|\gamma,x'-c|t'|\gamma')\frac{d\Omega d\Omega'}{(4\pi)^2} dxdx' \\
&= \frac{1}{(\frac{4}{3}\pi R^3)^2}\ tt'\int_{S \times S}\int_{\mathbb{R}^3 \times \mathbb{R}^3} \Bigg( f(x,x')k_{x_0}(x-c|t|\gamma,x'-c|t'|\gamma')\\
& \hspace{2cm} \times \mathbbm{1}_{[0,R]}(|x-c|t|\gamma - x_0|)\mathbbm{1}_{[0,R]}(|x'-c|t'|\gamma' - x_0|) \Bigg) dxdx'\frac{d\Omega d\Omega'}{(4\pi)^2}.
\end{align*}
The first indicator function restricts the integration domain of $x$ to $B(x_0 + c|t|\gamma,R)$, and symmetrically for the second indicator function and $x'$.
For $x$ in $B(x_0 + c|t|\gamma,R)$, in spherical coordinates around $x_0 + c|t|\gamma$, write $x = x_0 + c|t|\gamma + R \rho \gamma_x$ with $\rho \in [0,1]$, $\gamma_x \in S$ and associated surface differential element $d\Omega_x$. We do symmetrically for $x' \in B(x_0 + c|t'|\gamma',R)$, which yields
\begin{align*}
I_R &= tt'\int_{S \times S} \int_{S \times S} \int_0^1 \int_0^1 \Bigg(f(x_0 + c|t|\gamma + R \rho \gamma_x,x_0 + c|t'|\gamma' + R \rho' \gamma_{x'}) \\
& \hspace{3cm} \times  k(x_0 + R \rho \gamma_x,x_0 + R \rho' \gamma_{x'}) \Bigg) \times 9 \rho^2 d\rho \rho'^2 d\rho' \frac{d\Omega_x d\Omega_{x'}}{(4\pi)^2}\frac{d\Omega d\Omega'}{(4\pi)^2}.
\end{align*}
The integration domain above is a compact subset of $\mathbb{R}^{10}$. Since $f$ is continuous and $k$ is assumed continuous in the vicinity of $(x_0,x_0)$, Lebesgue's dominated convergence theorem can be applied when $R \rightarrow 0$, which yields
\begin{align*}
I_R \xrightarrow[R \rightarrow 0]{} &tt' k(x_0,x_0) \int_{S \times S} f(x_0 + c|t|\gamma,x_0 + c|t'|\gamma')\frac{d\Omega d\Omega'}{(4\pi)^2} \times \bigg(3\int_0^1\rho^2d\rho\bigg)^2 \\
& = k(x_0,x_0)\langle \tau_{x_0}F_t \otimes \tau_{x_0}F_{t'},f \rangle.
\end{align*}
which concludes the proof.
\end{proof} 
\begin{proof}[Proof of Proposition \ref{prop: log lik reg}]
Suppose first that $||F_{x_0}||_{\mathbb{R}^n}^2 = 0$. Then by definition, $r(x_0) = 0$ and $\mathcal{L}_{\mathrm{reg}}(x_0,\lambda) =||W||_{\mathbb{R}^n}^2/\lambda + n \log \lambda$ which indeed shows that 
\begin{align}\label{eq:pointwise}
\big|\lambda\mathcal{L}_{\mathrm{reg}}(x_0,\lambda) - ||W||_{\mathbb{R}^n}^2 \big| = O_{\lambda \rightarrow 0}(\lambda \log \lambda).
\end{align}
Now, let $\varepsilon > 0$ and assume that $||F_{x_0}||_{\mathbb{R}^n}^2 \geq \varepsilon$. 
We first deal with the first term in equation \eqref{eq:log lik reg}. Using the Sherman–Morrison formula (\cite{2007numerical_recipes}, Section 2.7.1), we may invert $(K_{x_0}^{\mathrm{reg}} + \lambda I_n)$ explicitly:
\begin{align*}
	(K_{x_0}^{\mathrm{reg}} + \lambda I_n)^{-1} = \frac{1}{\lambda}I_n - \frac{1}{\lambda^2} \frac{F_{x_0}F_{x_0}^T}{1+ \frac{1}{\lambda}F_{x_0}^TF_{x_0}} = \frac{1}{\lambda}\Big(I_n - \frac{F_{x_0}F_{x_0}^T}{\lambda + ||F_{x_0}||_{\mathbb{R}^n}^2}\Big).
\end{align*}
The determinant term in equation \eqref{eq:log lik reg} is also easily derived. Indeed, $F_{x_0}F_{x_0}^T$ has only one non zero eigenvalue equal to $||F_{x_0}||_{\mathbb{R}^n}^2$, since $(F_{x_0}F_{x_0}^T)F_{x_0} = F_{x_0}(F_{x_0}^TF_{x_0}) = ||F_{x_0}||_{\mathbb{R}^n}^2F_{x_0}$:
\begin{align}
\log \det(K_{x_0}^{\mathrm{reg}} + \lambda I_n) = (n-1) \log \lambda + \log (\lambda + ||F_{x_0}||_{\mathbb{R}^n}^2).
\end{align} 
(The same argument shows that $\rho(K_{x_0}^{\mathrm{reg}}) = ||F_{x_0}||_{\mathbb{R}^n}^2$.) Thus,
\begin{align*}
\mathcal{L}_{\mathrm{reg}}(x_0,\lambda) &= W^T (K_{x_0}^{\mathrm{reg}} + \lambda I_n)^{-1} W + \log \det(K_{x_0}^{\mathrm{reg}} + \lambda I_n) \nonumber \\
&= \frac{1}{\lambda}\bigg(||W||_{\mathbb{R}^n}^2 - \frac{\langle F_{x_0},W\rangle_{\mathbb{R}^n}^2}{\lambda + ||F_{x_0}||_{\mathbb{R}^n}^2}\bigg) + (n-1) \log \lambda + \log (\lambda + ||F_{x_0}||_{\mathbb{R}^n}^2) \nonumber \\
&= \frac{||W||_{\mathbb{R}^n}^2}{\lambda}\bigg(1 - \frac{\langle F_{x_0},W\rangle_{\mathbb{R}^n}^2}{||W||_{\mathbb{R}^n}^2(\lambda + ||F_{x_0}||_{\mathbb{R}^n}^2)}\bigg) + (n-1) \log \lambda + \log (\lambda + ||F_{x_0}||_{\mathbb{R}^n}^2).
\end{align*}
Therefore,
\begin{align}
\label{eq:log reg diff}
\lambda\mathcal{L}_{\mathrm{reg}}(x_0,\lambda)&-||W||_{\mathbb{R}^n}^2(1-r(x_0)^2) \nonumber \\
&= ||W||_{\mathbb{R}^n}^2\bigg(\frac{\langle F_{x_0},W \rangle_{\mathbb{R}^n}^2}{||W||_{\mathbb{R}^n}^2 ||F_{x_0}||_{\mathbb{R}^n}^2} - \frac{\langle F_{x_0},W \rangle_{\mathbb{R}^n}^2}{||W||_{\mathbb{R}^n}^2 (\lambda + ||F_{x_0}||_{\mathbb{R}^n}^2)} \bigg) \\ &\hspace{15pt}+ (n-1)\lambda \log \lambda + \lambda \log (\lambda + ||F_{x_0}||_{\mathbb{R}^n}^2). \nonumber
\end{align}
Moreover, for the term in equation \eqref{eq:log reg diff} which is multiplied by $||W||_{\mathbb{R}^n}^2$,
\begin{align}\label{eq:bound 1 log reg}
\frac{\langle F_{x_0},W \rangle_{\mathbb{R}^n}^2}{||W||_{\mathbb{R}^n}^2 ||F_{x_0}||_{\mathbb{R}^n}^2} - \frac{\langle F_{x_0},W \rangle_{\mathbb{R}^n}^2}{||W||_{\mathbb{R}^n}^2 (\lambda + ||F_{x_0}||_{\mathbb{R}^n}^2)} &= \frac{\langle F_{x_0},W \rangle_{\mathbb{R}^n}^2}{||W||_{\mathbb{R}^n}^2}\bigg(\frac{1}{||F_{x_0}||_{\mathbb{R}^n}^2} -\frac{1}{\lambda + ||F_{x_0}||_{\mathbb{R}^n}^2}  \bigg) \nonumber \\
=& \frac{\langle F_{x_0},W \rangle_{\mathbb{R}^n}^2}{||W||_{\mathbb{R}^n}^2} \frac{\lambda}{||F_{x_0}||_{\mathbb{R}^n}^2(\lambda + ||F_{x_0}||_{\mathbb{R}^n}^2)} \nonumber \\
\leq& r(x_0)^2  \frac{\lambda}{\lambda + ||F_{x_0}||_{\mathbb{R}^n}^2} \leq \frac{\lambda}{||F_{x_0}||_{\mathbb{R}^n}^2} \leq \frac{\lambda}{\varepsilon},
\end{align}
and obviously, since $\lambda \geq 0$, 
\begin{align}\label{eq:truc_positif}
\frac{\langle F_{x_0},W \rangle_{\mathbb{R}^n}^2}{||W||_{\mathbb{R}^n}^2 ||F_{x_0}||_{\mathbb{R}^n}^2} - \frac{\langle F_{x_0},W \rangle_{\mathbb{R}^n}^2}{||W||_{\mathbb{R}^n}^2 (\lambda + ||F_{x_0}||_{\mathbb{R}^n}^2)} \geq 0.
\end{align}
Also, one sees that $F_{x_0} = 0$ as soon as $\sup_i |x_0-x_i| > cT + R$, ie $x_0$ is too far from the receivers for them to capture non zero signal during the time interval $[0,T]$. Thus the function $x_0 \longmapsto ||F_{x_0}||_{\mathbb{R}^n}^2$ is zero outside of a compact set. It is obviously continuous on $\mathbb{R}^3$ and is thus bounded on $\mathbb{R}^3$ by some constant $M > 0$. Using this together with equations \eqref{eq:bound 1 log reg} and \eqref{eq:truc_positif} inside equation \eqref{eq:log reg diff}, and assuming that $\lambda \leq 1$ yields
\begin{align*}
\big|\lambda\mathcal{L}_{\mathrm{reg}}(x_0,\lambda)-||W||_{\mathbb{R}^n}^2(1-r(x_0)^2)\big| \leq \frac{\lambda}{\varepsilon}||W||_{\mathbb{R}^n}^2 + (n-1)|\lambda \log \lambda| + \lambda \log (M+1),
\end{align*}
which shows the uniform convergence statement as well as the pointwise one (together with \eqref{eq:pointwise}).
\end{proof}
\begin{proof}[Proof of Proposition \ref{prop: log lik reg N}]
In all concerned mathematical objects, we highlight the $N$ dependency with an exponent, i.e. $W^N$, $F_{x_0}^N$, etc.
We use the exact same tools as in the previous proof, namely that we the following equality holds:
\begin{align*}
\mathcal{L}_{\mathrm{reg}}^N(x_0,\lambda) &= \frac{||W^N||_{\mathbb{R}^n}^2}{\lambda}\bigg(1 - \frac{\langle F_{x_0}^N,W^N\rangle_{\mathbb{R}^n}^2}{||W^N||_{\mathbb{R}^n}^2\big(\lambda + ||F_{x_0}^N||_{\mathbb{R}^n}^2\big)}\bigg) \\
&\hspace{20pt} + (n-1) \log \lambda + \log (\lambda + ||F_{x_0}^N||_{\mathbb{R}^n}^2).
\end{align*}
But we also have $||W^N||_{\mathbb{R}^n}^2 = \sum_{i=1}^q \sum_{k=1}^N \tilde{w}(x_i,t_k)^2, ||F_{x_0}^N||_{\mathbb{R}^n}^2 = \sum_{i=1}^q \sum_{k=1}^N f_{t_k}^{\mathrm{R}}(x_i-x_0)^2$ and $\langle F_{x_0}^N,W^N\rangle_{\mathbb{R}^n} = \sum_{i=1}^q \sum_{k=1}^N f_{t_k}^{\mathrm{R}}(x_i-x_0)\tilde{w}(x_i,t_k)$.
Since the time steps are equally spaced, we can study the limit $N \rightarrow \infty$ of the above objects thanks to Riemann sums. When $N \rightarrow \infty$,
\begin{align}
\frac{1}{N}||W^N||_{\mathbb{R}^n}^2 &\longrightarrow \sum_{i=1}^q\int_0^T \tilde{w}(x_i,t)^2dt = ||I_{\mathrm{w}}||_{L^2}^2, \\
\frac{1}{N}||F_{x_0}^N||_{\mathbb{R}^n}^2 &\longrightarrow \sum_{i=1}^q\int_0^T f_t(x_i-x_0)^2dt = ||I_{x_0}||_{L^2}^2, \label{eq:I_x0_lim} \\
\frac{1}{N}
\langle W^N,F_{x_0}^N \rangle_{\mathbb{R}^n} &\longrightarrow \sum_{i=1}^q\int_0^T \tilde{w}(x_i,t) f_t(x_i-x_0)dt = \langle I_{\mathrm{w}},I_{x_0} \rangle_{L^2}.
\end{align}
Assume that $x_0$ is such that $||I_{x_0}||_{L^2} \neq 0$, then because of equation \eqref{eq:I_x0_lim}, the quantity $||F_{x_0}^N||_{\mathbb{R}^n}$ is bounded from below by a constant $C > 0$ for sufficiently large $N$ (say $C = ||I_{x_0}||_{L^2}/2$). From the three equations above, we then have the following convergence:
\begin{align}\label{eq:cv_riemann}
\frac{\langle F_{x_0}^N,W^N\rangle_{\mathbb{R}^n}^2}{||W^N||_{\mathbb{R}^n}^2(\lambda + ||F_{x_0}^N||_{\mathbb{R}^n}^2)} = \frac{(\frac{1}{N}\langle F_{x_0}^N,W^N\rangle_{\mathbb{R}^n})^2}{\frac{1}{N}||W^N||_{\mathbb{R}^n}^2(\frac{\lambda}{N} + \frac{1}{N} ||F_{x_0}^N||_{\mathbb{R}^n}^2)}  \xrightarrow[N \rightarrow \infty]{} r_{\infty}(x_0).
\end{align}
Likewise, since $n = qN$, when $N \rightarrow \infty$ we have that
\begin{align*}
\frac{(n-1)\log \lambda}{N} &+ \frac{1}{N}\log (\lambda + ||F_{x_0}||_{\mathbb{R}^n}^2) \nonumber \\
&= \frac{(Nq-1)\log \lambda}{N} + \frac{\log N}{N} + \frac{1}{N}\log \Big(\frac{\lambda}{N} + \frac{1}{N}||F_{x_0}||_{\mathbb{R}^n}^2\Big)  \xrightarrow[N \rightarrow \infty]{} q \log \lambda.
\end{align*}
which, together with equation \eqref{eq:cv_riemann}, shows the announced result.
\end{proof}
\begin{proof}[Proof of Proposition \ref{prop: Lp control}]
We have $(F_t*v_0)(x) = t\int_{S}v_0(x-c|t|\gamma)d\Omega/4\pi$, where 
$d\Omega/4\pi$ is the normalized Lebesgue measure on the unit sphere $S$. Assume first that $p\in [1,+\infty [$. Jensen's inequality on the function $t \longmapsto |t|^p$ yields
\begin{align}
||F_t * v_0||_p^p &= t^p \int_{\mathbb{R}^3}|(F_t*v_0)(x)|^p dx = |t|^p\int_{\mathbb{R}^3}\bigg|\int_{S}v_0(x-c|t|\gamma)\frac{d\Omega}{4\pi}\bigg|^p dx \nonumber\\
 &\leq |t|^p\int_{\mathbb{R}^3} \int_{S}|v_0(x-c|t|\gamma)|^p \frac{d\Omega}{4\pi} dx = |t|^p \int_{S} \int_{\mathbb{R}^3} |v_0(x-c|t|\gamma)|^p dx \frac{d\Omega}{4\pi} \nonumber\\
 &\leq \int_{S} ||v_0||_p^p \frac{d\Omega}{4\pi} = |t|^p||v_0||_p^p \label{eq:lp_control_first_estimate_proof},
\end{align}
which yields equation \eqref{eq:estim Lp v_0}. Next,
\begin{align*}
(\dot{F}_t *u_0)(x) &= \partial_t (F_t * u_0)(x) = \partial_t \Bigg( t\int_{S}u_0(x-c|t|\gamma)\frac{d\Omega}{4\pi} \Bigg) \\
&= \int_{S}u_0(x-c|t|\gamma)\frac{d\Omega}{4\pi} + t\int_{S}- c \gamma \cdot \nabla u_0(x-c|t|\gamma)\frac{d\Omega}{4\pi}
&=: I_1(x) + I_2(x).
\end{align*}
The functions $I_1$ and $I_2$ are defined in the equation above. We have $||\dot{F}_t *u_0||_p = ||I_1 + I_2||_p \leq ||I_1||_p + ||I_2||_p$.
As in \eqref{eq:lp_control_first_estimate_proof}, $||I_1||_p \leq ||u_0||_p$. From Jensen's inequality,
\begin{align*}
||I_2||_p^p = |ct|^p\int_{\mathbb{R}^3} \Bigg|\int_{S} \gamma \cdot \nabla u_0(x-c|t|\gamma)\frac{d\Omega}{4\pi} \Bigg|^p dx \leq |ct|^p \int_{\mathbb{R}^3}\int_{S} |\gamma \cdot \nabla u_0(x-c|t|\gamma)|^p \frac{d\Omega}{4\pi}dx.
\end{align*}
Next, we use Hölder's inequality in $\mathbb{R}^3$: $|\gamma \cdot \nabla u_0| \leq |\nabla u_0|_p \times |\gamma|_q$ with ${1}/{p} + {1}/{q} = 1$, where $|v|_p = (|v_1|^p + |v_2|^p + |v_3|^p)^{1/p}$  and likewise for $|v|_q$. Thus,
\begin{align*}
||I_2||_p^p &\leq c^p |t|^p \int_{\mathbb{R}^3}\int_{S} |\nabla u_0(x-c|t|\gamma)|_p^p \times |\gamma|_q^p \frac{d\Omega}{4\pi} dx \\
&\leq c^p |t|^p\int_{S}|\gamma|_q^p \int_{\mathbb{R}^3} |\nabla u_0(x-c|t|\gamma)|_p^p dx \frac{d\Omega}{4\pi} = c^p |t|^p\Bigg(\int_{S}|\gamma|_q^p \frac{d\Omega}{4\pi}\Bigg) ||\nabla u_0||_p^p.
\end{align*}
which yields equation \eqref{eq:estim Lp v_0}. Finally, the case $p = +\infty$ is trivial.
Equation \eqref{eq:diff w m tilde} is then the result of equations \eqref{eq:estim Lp v_0} and  \eqref{eq:estim Lp u_0} applied to the function 
\begin{align*}
w(x,t) - \tilde{m}(x,t) = [F_t *(v_0 - \tilde{v}_0)](x) + [\dot{F}_t *(u_0 - \tilde{u}_0)](x) .
\end{align*}
This finishes the proof.
\end{proof}
\end{document}